\documentclass[12pt,a4paper]{article}
\usepackage[left=1.5cm, right=1.5cm]{geometry}%%%% Remove this before submitting
\usepackage{enumerate, tikzsymbols,faktor}
\usepackage{amssymb, comment,mathtools}
\usepackage{tikz-cd}
\usepackage{mdframed,lipsum}
\tikzset{every picture/.append style={remember picture}}
\usepackage{amsmath, hyperref}
\usepackage{latexsym,appendix}
\usepackage{amsthm}
\usepackage{xypic}
\usepackage{amsmath}
\usepackage{tikz}
\usetikzlibrary{matrix, arrows}

\newtheorem{theorem}{Theorem}
\newtheorem{lemma}[theorem]{Lemma}
\newtheorem{proposition}[theorem]{Proposition}
\newtheorem{cor}{Corollary}
\theoremstyle{definition}
\newtheorem{definition}[theorem]{Definition}

\newtheorem{con}{Conjecture}
\theoremstyle{remark}
\newtheorem{remark}[theorem]{Remark}

\newtheorem{example}[theorem]{Example}

\newtheorem{claim}[]{Claim}

\numberwithin{equation}{section}

\usepackage[english]{babel}

\newcommand{\Q}{\mathbb{Q}}

\makeatletter
\newcounter{author}
\renewcommand*\author[1]{%
  \stepcounter{author}%
  \ifnum\c@author=1
    \gdef\@author{#1}%
  \else
    \xdef\@author{\unexpanded\expandafter{\@author\and#1}}%
  \fi
  \csgdef{author@\the\c@author}{#1}}
\newcommand*\email[1]{%
  \csgdef{email@\the\c@author}{#1}}
\newcommand*\address[1]{%
  \csgdef{address@\the\c@author}{#1}}
\AtEndDocument{%
  \xdef\author@count{\the\c@author}%
  \c@author=1
  \print@authors}
\newcommand*\print@authors{%
  \ifnum\c@author>\author@count
  \else
    \print@author{\the\c@author}%
    \advance\c@author by 1
    \expandafter\print@authors
  \fi}
\newcommand*\print@author[1]{%
  \par\medskip
  \begin{tabular}{@{}l@{}}%
    \textsc{Addresses:}\\
    \csuse{address@#1}
  \end{tabular}}
\makeatother

\begin{document}

\title{New evidence for R\'emond's generalisation of Lehmer's conjecture}
\author{S. Checcoli and G. A. Dill}
\address{S. Checcoli: \\ Univ. Grenoble Alpes, CNRS, IF, 38000 Grenoble, France. \\ \href{mailto: sara.checcoli@univ-grenoble-alpes.fr}{sara.checcoli@univ-grenoble-alpes.fr}\\
G. A. Dill: \\ Institut de Math\'ematiques, Universit\'e de Neuch\^atel,\\ Rue Emile-Argand 11, 2000 Neuch\^atel, Switzerland.\\
\href{mailto: gabriel.dill@unine.ch}{gabriel.dill@unine.ch}}
\date{\today}
\newcommand\extrafootertext[1]{%
    \bgroup
    \renewcommand\thefootnote{\fnsymbol{footnote}}%
    \renewcommand\thempfootnote{\fnsymbol{mpfootnote}}%
    \footnotetext[0]{#1}%
    \egroup
}
\maketitle
\extrafootertext{2020 \textit{Mathematics Subject Classification}. 20K15, 11J95, 11R32, 11G50. Keywords: free groups, division closure, saturated closure, Kummer theory, tori, abelian varieties, almost split semiabelian varieties, height lower bound.}

\abstract{In this article, we generalise a result of Pottmeyer from the multiplicative group of the algebraic numbers to almost split semiabelian varieties defined over number fields. This concerns a consequence of R\'emond's generalisation of Lehmer's conjecture. Namely, for a finite rank subgroup $\Gamma$ of an almost split semiabelian variety $G$, we consider the group of rational points of $G$ over a finite extension of the field generated by the saturated closure of $\Gamma$, i.e. the division closure of the subgroup generated by $\Gamma$ and all its images under geometric endomorphisms of $G$. We show that this becomes a free group after one quotients out the saturated closure of $\Gamma$. The proof uses, amongst other ingredients, a criterion of Pottmeyer, which relies on a result of Pontryagin, together with a result from Kummer theory, of which we reproduce a proof by R\'emond.}

\section{Introduction}\label{intro}
Let $K$ be a number field and let $\overline{K}$ be a fixed algebraic closure of $K$. 
Let $G$ be a semiabelian variety defined over $K$. Recall that an \emph{isogeny} is a surjective algebraic group homomorphism whose kernel has finite order, called the \emph{degree} of the isogeny. We call two semiabelian varieties \emph{isogenous} if there exists an isogeny between their base changes to $\overline{K}$ (this defines an equivalence relation). We say that $G$ is \emph{split}
if $G$ is isomorphic over $K$ to $\mathbb{G}_m^t \times A$ where  $\mathbb{G}_m=\mathbb{G}_{m,K}$ is the one-dimensional torus over $K$, $t \in \mathbb{Z}_{\geq 0}$, and $A$ is an abelian variety defined over $K$. We say that $G$  is \emph{almost split} if it is isogenous to a split semiabelian variety or, equivalently, to a product of the form $\mathbb{G}_m^t \times A$ as above.
In this article we use the additive notation for the group law on $G$, denoting by $0$ the corresponding neutral element and by $\mathrm{End}(G)$ the ring of all algebraic group endomorphisms of the base change of $G$ to $\overline{K}$.

In \cite{Rem}, Rémond proposes a framework that unifies various statements concerning height lower bounds on tori and abelian varieties, which stem from Lehmer's conjecture. The latter, originally formulated for the group $\overline{\mathbb{Q}}^\times$, asserts that the product of the logarithmic Weil height of an algebraic number and its degree cannot be arbitrarily small without being zero. In his article, Rémond puts forward a conjecture \cite[Conjecture 3.4]{Rem} that significantly generalises Lehmer's conjecture and its so-called relative form. We discuss this in detail in Appendix \ref{Conjecture-1-rank0} and we also refer to \cite{Rem}, particularly Section 3 therein, for a detailed survey on these problems. 

Our main motivation for this article is a very special case of this conjecture. Before stating it, we need to introduce some notation.

We denote by $h$ be the Weil height on $\mathbb{G}_m^t$, that is, the function assigning to a point the sum of the (logarithmic) Weil heights of its coordinates. We also let $h_{\mathcal{L}}$ denote the N\'eron-Tate height on $A$ attached to an ample symmetric line bundle $\mathcal{L}$ on $A$ (see for instance \cite[Section 9.2]{BG} for more details).
We define the canonical height $h_G=h_{G,\mathcal{L}}$ on $G=\mathbb{G}_m^t\times A$ by $h_G(P_T,P_A) = h(P_T) + h_{\mathcal{L}}(P_A)$ for any point $(P_T,P_A) \in \mathbb{G}_m^t(\overline{K}) \times A(\overline{K}) = G(\overline{K})$.
Clearly, $h_G=h$ when $G=\mathbb{G}_m^t$ and $h_G=h_{\mathcal{L}}$ when $G=A$.

If $\Gamma$ is a subgroup of $G(\overline{K})$, its rank is, as usual, the cardinality of a maximal $\mathbb{Z}$-linearly independent subset (which is equal to the dimension of the $\mathbb{Q}$-vector space $\Gamma \otimes_{\mathbb{Z}} \mathbb{Q}$). We also denote by $\mathrm{End}(G)\cdot \Gamma$
the subgroup of $G(\overline{K})$ generated by all the elements of the form $\varphi(\gamma)$ with $\varphi\in \mathrm{End}(G)$ and $\gamma\in \Gamma$.  A central object in this article is $\Gamma_{\mathrm{sat}}$, the \emph{saturated closure of $\Gamma$}, defined as 
the group of points $P\in G(\overline{K})$ such that $nP\in \mathrm{End}(G)\cdot \Gamma$ for some $n\in \mathbb{N}\setminus\{0\}$.  For instance, if $\Gamma$ has rank zero, $\Gamma_{\mathrm{sat}}=G_{\mathrm{tor}}$, where  $G_{\mathrm{tor}}$ is the group of torsion points of $G$, while if $\mathrm{End}(G)=\mathbb{Z}$, $\Gamma_{\mathrm{sat}}$ is the so-called \emph{division closure} of $\Gamma$. If $\dim G > 0$, then the group $\Gamma_{\mathrm{sat}}$ always contains infinitely many points of height zero and if $\Gamma$ has positive rank, then $\Gamma_{\mathrm{sat}}$ also contains points of arbitrarily small non-zero height. 

We state here a very special case of R\'emond's conjecture \cite[Conjecture 3.4]{Rem}, as formulated in \cite[Conjecture 1.1(c)]{Pot} (for a proof that this conjecture follows from R\'emond's conjecture, see Appendix \ref{app-rem-conj}):
\begin{con}[R\'emond]\label{Rem-conj} 
Let $G$ be either $\mathbb{G}_m^t$ for some $t \in \mathbb{Z}_{\geq 0}$ or an abelian variety defined over a number
field $K$. Let $\Gamma$ be a subgroup of $G(\overline{K})$ of finite rank.
For every finite extension $L/K(\Gamma_{\mathrm{sat}})$ there is a constant $c_L>0$ such that
$h_G(P)\geq c_L$ for every $P\in G(L)\setminus \Gamma_{\mathrm{sat}}$.
\end{con}
This conjecture embodies the idea that in a field where small points (i.e. points of small height) can reasonably be found, the only small points are the ones you might expect.

More concretely, if  $G = \mathbb{G}_m$ and $K = \mathbb{Q}$, one way to construct small points in $G(\overline{\mathbb{Q}})$ is to take roots of a fixed algebraic number of growing degree.

For example, if $L$ is the field obtained by adjoining to $\mathbb{Q}$ all $n$-th roots of $2$ for $n$ growing larger and larger, then $L=\mathbb{Q}(\Gamma_{\mathrm{sat}})$ for  $\Gamma = \langle 2 \rangle$. Clearly, $L$ contains many points of ``small" height, for instance all points of the form $\zeta 2^{a/n}$, where $\zeta$ is any root of unity and $|a|< n$. On the other hand, these are precisely the points of $\Gamma_{\mathrm{sat}}$ of height less than $ \log 2$, so, according to the conjecture, there is some $c > 0$ such that every $\alpha \in L^\times$ with $h(\alpha) < c$ is of this form.

Conjecture \ref{Rem-conj} is widely open, the only known cases being for $\Gamma$ of rank zero. In this case, it is equivalent to the statement that, for all $K$ and $G$, the field $K(G_{\mathrm{tor}})$ has the so-called property (B) (with respect to $G$), i.e. that $G(K(G_{\mathrm{tor}}))$ contains no points of arbitrarily small positive height (see \cite{BZ01} for the original definition of property (B)). This has been proved for instance when $G$ is a torus (as a consequence of \cite{AZ}) or an abelian variety with CM (as a consequence of \cite{BS}).

For $\Gamma$ of positive rank, Conjecture \ref{Rem-conj} is fully open even in the simplest cases. For example, in the above-mentioned case where $G=\mathbb{G}_m$ and $\Gamma=\langle 2\rangle$, the most significant partial result towards Conjecture \ref{Rem-conj} has been obtained in \cite{Amo}, where it is proved, in particular, that the height of any element of the set $\mathbb{Q}(\zeta_3,2^{1/3},\zeta_{3^2},2^{1/3^2},\ldots)^\times \setminus \Gamma_{\mathrm{sat}}$ is bounded from below by $\log(3/2)/18$ (here, $\zeta_n$ denotes a primitive $n$-th root of unity). A non-effective elliptic analogue of this result has been proved in \cite{Ple24}.
A possible approach to the toric case of Conjecture \ref{Rem-conj} has been discussed in \cite{Ple1}.

Even though the analogue of Conjecture \ref{Rem-conj} is false for general semiabelian varieties because of the so-called \emph{deficient} or \emph{Ribet points} (see \cite[Section 5]{Rem}), Plessis proposed a variant of the conjecture for general semiabelian varieties \cite[Conjecture 1.3]{Ple2}. He proved moreover that his variant specialises to the analogue of Conjecture \ref{Rem-conj} if $G$ is a product of a power of $\mathbb{G}_m$ and an abelian variety (because in this case the deficient points are precisely the torsion points).
In Corollary \ref{cor:torus-x-cm} (see Appendix \ref{Conjecture-1-rank0}), we show how to deduce from the literature this analogue of Conjecture \ref{Rem-conj} when $\Gamma$ has rank zero and $G$ is the product of a torus and a CM abelian variety (notice that the case $G=\mathbb{G}_{m, \mathbb{Q}} \times E$ where $E$ is an elliptic curve defined over $\mathbb{Q}$ was already proved in \cite[Corollary 1.1]{Hab13}). 
Notice also that, in particular, Plessis' variant of Conjecture \ref{Rem-conj} implies that, for an arbitrary abelian variety $A$ defined over $\mathbb{Q}$, the field $\mathbb{Q}(A_{\mathrm{tor}})$ does not contain elements of arbitrarily small positive Weil height, which is an open problem that seems to be out of reach at the moment.

In \cite[Proposition 1.2 and its proof]{Pot}, Pottmeyer exhibits a necessary condition for the validity of Conjecture \ref{Rem-conj} in such a more general setting.
\begin{proposition}[Pottmeyer]\label{prop-pott} Let $G$ be the product of a power of $\mathbb{G}_m$ and an abelian variety defined over a number field $K$. Let $\Gamma$ be a finite rank subgroup of $G(\overline{K})$.
If the analogue of Conjecture \ref{Rem-conj} is true for $G$, then, for every finite extension $L/K(\Gamma_{\mathrm{sat}})$, the group $G(L)/\Gamma_{\mathrm{sat}}$ is free abelian.
\end{proposition}
The proof of this result (see \cite[Lemma 2.6]{Pot}) combines two logical equivalences: the first one is that an abelian group is free if and only if it admits a discrete norm; the second one is that Conjecture \ref{Rem-conj} is equivalent to the statement that a specific norm on $G(L)/\Gamma_{\mathrm{sat}}$ is discrete. Precursors of Proposition \ref{prop-pott} are \cite[Corollary 1.2]{Hab13} and \cite[Proposition 1.1]{GHP15}.

It follows from \cite[Lemma A.7 and Remark A.8]{BHP} that the necessary condition of Proposition \ref{prop-pott} is satisfied when $\Gamma$ has rank zero, that is when $\Gamma_{\mathrm{sat}}=G_{\mathrm{tor}}$.
In \cite[Theorem 1.3]{Pot}, Pottmeyer proves that this is still the case for groups of arbitrary rank when $G=\mathbb{G}_m$. 

Our main result shows that the necessary condition as in Proposition \ref{prop-pott} is always satisfied by almost split semiabelian varieties, generalising both \cite[Lemma A.7]{BHP} and \cite[Theorem 1.3]{Pot}.
More precisely, we prove the following:
\begin{theorem}\label{cor-main}
Let $G$ be an almost split semiabelian variety over a number field $K$.  Let $\Gamma \subset G(\overline{K})$  be a subgroup of finite rank and let $L/K(\Gamma_{\mathrm{sat}})$ be a finite extension. Then the group $G(L)/\Gamma_{\mathrm{sat}}$ is free abelian.
\end{theorem}
A natural question is whether the stronger conclusion that $G(L)/\Gamma_{\mathrm{sat}}$ is free of infinite rank holds in the setting of Theorem \ref{cor-main}. This is clearly the case if 
the toric part of $G$ is positive-dimensional and Habegger asked whether the same holds if $G$ is an abelian variety. We do not know the answer, but note that a long-standing (and to our knowledge open) problem of Frey and Jarden is to decide whether $A(K^{\mathrm{ab}})$ has infinite rank for every abelian variety $A$ over a number field $K$, where $K^{\mathrm{ab}}$ denotes the maximal abelian extension of $K$ (see \cite{FJ74}). In the same article, Frey and Jarden showed that this holds for every elliptic curve $E$ over $\mathbb{Q}$, which implies that $E(\mathbb{Q}(\Gamma_{\mathrm{sat}}))/\Gamma_{\mathrm{sat}}$ is indeed free of infinite rank (see also \cite[Corollary 1.2]{Hab13} for the case where $\Gamma_{\mathrm{sat}} = E_{\mathrm{tor}}$).

The proof of Theorem \ref{cor-main} relies on several reduction steps. The first step, done in Section \ref{red-step-1}  (see Proposition \ref{reduction-step}), shows that the validity of Theorem \ref{cor-main} is established by proving it in the case where $G = \mathbb{G}_m^t \times A$, $A$ is an abelian variety such that \emph{all geometric endomorphisms of $A$ are defined over $K$} (i.e. all endomorphisms of the abelian variety obtained by base-changing $A$ to  $\overline{K}$ are already defined over $K$), and $\Gamma\subset G(K)$ is finitely generated. This step relies only on basic facts about algebraic groups and abelian varieties. The second step, detailed in Section \ref{reduction-second-step} (see Proposition \ref{theorem3-implies-theorem2}), combines a criterion of Pottmeyer in \cite[Lemma 3.1]{Pot} (recalled in Lemma \ref{Lukas-with-F}), concerning the freeness of abelian groups and based on a result of Pontryagin, with the result \cite[Proposition 3.3]{Pot} (recalled in Proposition \ref{Pot-prop3.3}) to further reduce the problem to showing the following result, which represents the main technical contribution of our work. 

\begin{theorem}\label{lem:arbitraryrank-var-ab}
Let $G = \mathbb{G}_{m,K}^{t} \times A$, where $K$ is a number field, $t \in \mathbb{Z}_{\geq 0}$, and $A$ is an abelian variety over $K$ having all geometric endomorphisms defined over $K$. Let $\Gamma \subset G(K)$ be a finitely generated subgroup. Then there exists a positive integer $d= d(G,\Gamma,K)$ such that, for any $P \in G(K(\Gamma_{\mathrm{sat}})) \cap G(K)_{\mathrm{div}}$, we have $d P \in G(K)+\Gamma_{\mathrm{sat}}$.
\end{theorem}
When $\Gamma$ has rank zero, this result essentially follows from
\cite[Appendix A]{BHP} (as explained in the proof of \cite[Theorem 3.2]{Pot}). To prove Theorem \ref{lem:arbitraryrank-var-ab} for $\Gamma$ of arbitrary finite rank, we proceed as follows: first, we show that both $\mathrm{End}(G)$ and $\Gamma_{\mathrm{sat}}$ have some ``split structure'' (see Section \ref{splitting-prelim}). This allows us to work on the point $P$ coordinatewise on the toric and the abelian factor. However, we cannot study the different coordinates completely separately from each other as the field $K(\Gamma_{\mathrm{sat}})$ depends on all factors at once. To find an integer $d$ that works for every coordinate (and that is independent of $P$), we first show that we can find a finite set of points on $G(K)$ that, in some sense, ``generate" $\Gamma_{\mathrm{sat}}$ and enlarge it to a finite set of points that ``generate" a group (with some useful properties) containing the coordinates of our fixed point $P$. For the abelian factor, this is trivial because of the Mordell-Weil theorem, but for the toric factor, we need to do some work in Section \ref{prel-gen-Gamma-sat}. We then show in Section \ref{sec-end-of-proof} that we can essentially express the coordinates of $P$ as $\mathbb{Z}$-linear combinations of division points of such points. A classical result from Kummer theory (Theorem \ref{main-GR} in Section \ref{sec-Kummer}), of which we reproduce a proof by R\'emond in Appendix \ref{sec-proof-GR-main}, allows us to complete the proof giving a uniform bound for $d$. We remark that also the proofs of the above-mentioned results from \cite{BHP,Pot} made use of  Kummer theory. However, in our more general context, we need a stronger, more precise statement and its application is more involved.

In Appendix \ref{Conjecture-1-rank0}, we detail how R\'emond's generalisation of Lehmer's conjecture
\cite[Conjecture 3.4]{Rem} implies Conjecture \ref{Rem-conj}. We also show how to combine known results towards the relative form of Lehmer's problem to deduce the validity of Conjecture \ref{Rem-conj} for groups of rank 0 in certain semiabelian varieties.

\section{Preliminaries on commutative algebraic groups}\label{sec-prelim-alg-groups}
\subsection{On division and saturated closures}
Let $K$ be a number field and let $\overline{K}$ be a fixed algebraic closure of $K$. 
Let $G$ be a commutative algebraic group defined over $K$.
We recall that we use the additive notation for the group law on $G$, denoting by $0$ the corresponding neutral element and by $\mathrm{End}(G)$ the ring of all algebraic group endomorphisms of $G$ defined over $\overline{K}$. We write $\mathrm{End}_K(G)$ for the ring of  algebraic group endomorphisms of $G$ that are defined over $K$.

For a positive integer $n$ and a point $P\in G(\overline{K})$, we denote by $nP$ the image of $P$ under  multiplication by $n$ on $G$. Furthermore, we denote by \[G[n] = \{P \in G(\overline{K})\mid nP = 0\}\] the group of torsion points of $G$ of order dividing $n$ and by
\[G_{\mathrm{tor}}=\bigcup_{n\in \mathbb{Z}_{>0}} G[n]\]
 the torsion subgroup of $G(\overline{K})$. 
 We further set \begin{equation}\label{G_K_tor}G(K)_\mathrm{tor}=G_{\mathrm{tor}}\cap G(K).\end{equation}
If $S\subset G(\overline{K})$ is a set of points, we denote by \begin{equation}\label{def-End(A)-set}\mathrm{End}(G)\cdot S \quad \text{and} \quad \mathrm{End}_K(G)\cdot S\end{equation}
the subgroups of $G(\overline{K})$ generated by all the elements of the form $\varphi(x)$ with $x\in S$ and $\varphi\in \mathrm{End}(G)$ or $\varphi \in \mathrm{End}_K(G)$ respectively. If $S = \{P\}$ is a singleton, we also write $\mathrm{End}(G)\cdot P$ and $\mathrm{End}_K(G)\cdot P$ respectively.

We are now going to define some objects that will be central to our work.
\begin{definition}\label{def-Gamma-div-sat}
Let $\Gamma$ be a subgroup of $G(\overline{K})$. We define:
\begin{enumerate}[(i)]
\item the \emph{division closure of $\Gamma$} as
\[\Gamma_{\mathrm{div}} =\bigcup_{n\in \mathbb{N}\setminus\{0\}} \{P \in G(\overline{K})\mid nP \in \Gamma\};\]
\item the \emph{saturated closure of $\Gamma$} as
 \[\Gamma_{\mathrm{sat}} = (\mathrm{End}(G)\cdot\Gamma)_{\mathrm{div}},\]
 where $\mathrm{End}(G)\cdot\Gamma$ denotes the subgroup of $G(\overline{K})$ defined in \eqref{def-End(A)-set}. Equivalently, $P\in \Gamma_{\mathrm{sat}}$ if and only if there exist $n\in\mathbb{N}\setminus \{0\}$, $\varphi_1,\ldots,\varphi_r\in \mathrm{End}(G)$, and $\gamma_1,\ldots,\gamma_r\in \Gamma$ such that $nP=\varphi_1(\gamma_1)+\cdots+\varphi_r(\gamma_r)$. 
 \end{enumerate}
\end{definition}

\begin{example}
If $\Gamma=\{0\}$, then $\Gamma_{\mathrm{div}}=G_{\mathrm{tor}}$, while, if $\mathrm{End}(G)=\mathbb{Z}$, then $\mathrm{End}(G)\cdot\Gamma=\Gamma$ and $\Gamma_{\mathrm{sat}}=\Gamma_{\mathrm{div}}$.
\end{example}
\begin{remark}\label{rmk:sat-div-sat}
    We have, in general, that \begin{equation}\label{Gamma-sat-div}\Gamma_{\mathrm{sat}}=\left(\Gamma_{\mathrm{sat}}\right)_{\mathrm{div}}=\left(\Gamma_{\mathrm{div}}\right)_{\mathrm{sat}} = \left(\Gamma_{\mathrm{sat}}\right)_{\mathrm{sat}}.
    \end{equation}
\end{remark}
\begin{proof}
    It is clear that $\Gamma_{\mathrm{sat}}\subset \left(\Gamma_{\mathrm{div}}\right)_{\mathrm{sat}} \subset \left(\Gamma_{\mathrm{sat}}\right)_{\mathrm{sat}}$ and $\Gamma_{\mathrm{sat}}\subset \left(\Gamma_{\mathrm{sat}}\right)_{\mathrm{div}} \subset \left(\Gamma_{\mathrm{sat}}\right)_{\mathrm{sat}}$.

    Moreover, if $P \in \left(\Gamma_{\mathrm{sat}}\right)_{\mathrm{sat}}$, then $nP \in \mathrm{End}(G) \cdot \Gamma_{\mathrm{sat}}$ for some integer $n \geq 1$, so
    \[ nP =\varphi_1(\gamma_1')+\cdots+\varphi_r(\gamma_r'),\]
    where $\varphi_i\in \mathrm{End}(G)$ and $\gamma_i'\in \Gamma_{\mathrm{sat}}$ ($i = 1, \hdots, r$). For every $1\leq i\leq r$, there is an integer $m_i\geq 1$ such that $m_i\gamma_i'\in \mathrm{End}(G) \cdot \Gamma$. Set $m=nm_1\cdots m_r$, then
    \[ mP = \varphi_1(m_1\cdots m_r \gamma_1')+\cdots+\varphi_r(m_1\cdots m_r \gamma_r') \in \mathrm{End}(G) \cdot (\mathrm{End}(G) \cdot \Gamma) = \mathrm{End}(G) \cdot \Gamma\]
    and so $P \in \Gamma_{\mathrm{sat}}$. This means that $\left(\Gamma_{\mathrm{sat}}\right)_{\mathrm{sat}} \subset \Gamma_{\mathrm{sat}}$ and we are done.
\end{proof}
\subsection{On homomorphisms of algebraic groups}
In this section we collect two classical lemmas about homomorphisms of algebraic groups that will be used in our proofs.
\begin{lemma}\label{fact-isogeny-degree}
Let $G, G'$ be two connected commutative algebraic groups (for example, semiabelian varieties) defined over a field $K$ of characteristic $0$ and let $\varphi: G \to G'$ be an isogeny of degree $m > 0$, defined over $K$. Then there exists an isogeny $\psi: G' \to G$, also defined over $K$, such that the morphisms $\psi\circ \varphi$ and $\varphi\circ \psi$ are the multiplication by $m$ on $G$ and $G'$ respectively.
\end{lemma}

\begin{proof}
The lemma follows from \cite[Theorem 5.13]{Milne} and the fact that $G[m]$ is finite. 
\end{proof}

\begin{lemma} \label{lem:section-homomorphism} 
Let $f: B \to B'$ be a homomorphism of abelian varieties, defined over a field $L$ of characteristic $0$. Then there exist an integer $N \geq 1$ and a homomorphism of abelian varieties $\varphi: B' \to B$, also defined over $L$, such that
\[ f \circ \varphi \circ f = f \circ [N] = [N] \circ f,\]
where, by abuse of notation, $[N]$ denotes the multiplication-by-$N$ endomorphisms of $B$ and $B'$ respectively.
\end{lemma}

\begin{proof}
Let $B_1 \subset B$ denote the identity component of the kernel of $f$. By Poincar\'e's complete reducibility theorem \cite[Theorem 1 on p.~173]{Mumford}, there exists an abelian subvariety $B_2$ of $B$, defined over $L$, such that the restriction of the addition homomorphism $B_1 \times B_2 \to B$ is an isogeny. In particular, $f$ induces an isogeny from $B_2$ onto $f(B)$. By Lemma \ref{fact-isogeny-degree}, there exist a homomorphism $\varphi'_1: f(B) \to B_2$, defined over $L$, and an integer $N_1 \geq 1$ such that $(f \circ \varphi'_1)(P) = N_1P$ for all $P \in f(B)(L)$.

Applying Poincar\'e's complete reducibility theorem once more, we find that there exists an abelian subvariety $B_1'$ of $B'$, defined over $L$, such that the restriction of the addition homomorphism $f(B) \times B_1' \to B'$ is an isogeny. It follows from Lemma \ref{fact-isogeny-degree} that there exist a homomorphism $\varphi'_2: B' \to f(B)$, defined over $L$, and an integer $N_2 \geq 1$ such that $\varphi'_2(P') = N_2P'$ for all $P' \in f(B)(L)$.

Setting $\varphi = \varphi'_1 \circ \varphi_2'$ and $N = N_1N_2$, we find that
\begin{align*}
    (f \circ \varphi \circ f)(P) = (f \circ \varphi'_1)(\varphi'_2(f(P))) 
    = (f \circ \varphi'_1)(N_2f(P))
    = Nf(P)
\end{align*}
for all $P \in B(L)$ as desired.
\end{proof}

\section{A first reduction}\label{red-step-1}

The goal of this section is to prove the following first reduction step:

\begin{proposition}\label{reduction-step}
     The validity of Theorem \ref{cor-main} is established by proving it in the special case where $G$ and $\Gamma$ satisfy the following conditions:
    \begin{enumerate}[(i)]
    \item $G$ is the product of a power of the multiplicative group $\mathbb{G}_{m,K}$ and an abelian variety defined over $K$ having all geometric endomorphisms  defined over $K$;
    \item $\Gamma$ is a finitely generated subgroup of $G(K)$.
    \end{enumerate}
\end{proposition}

To prove Proposition \ref{reduction-step} we need the following lemma.

\begin{lemma}\label{lem-deduction}
Let $K$ be a number field with a fixed algebraic closure $\overline{K}$ and let $G$ be a connected commutative algebraic group over $K$. Suppose $\Gamma$ is a subgroup of $G(\overline{K})$ of finite rank. Then there exists a finite extension $E/K$ with $E\subset K(\Gamma_{\mathrm{sat}})$ and a finitely generated subgroup $\Gamma'$ of $G(E)$ such that $\Gamma_{\mathrm{sat}} = \Gamma'_{\mathrm{sat}}$. 
\end{lemma}
\begin{proof} Let $r$ be the rank of $\Gamma$ and take a maximal $\mathbb{Z}$-linearly independent subset $\{\gamma_1, \dots, \gamma_r\}$ of $\Gamma$.  
Set $\Gamma' = \mathbb{Z}\gamma_1 + \dots + \mathbb{Z}\gamma_r$ and $E=K(\gamma_1, \ldots, \gamma_r)$. 

Clearly, $E$ is a finite extension of $K$  contained in $K(\Gamma_{\mathrm{sat}})$ and $\Gamma'$ is a finitely generated subgroup of $G(E)$.
Moreover, by construction, $\Gamma'_{\mathrm{div}} = \Gamma_{\mathrm{div}}$. Hence, by \eqref{Gamma-sat-div}, 
\[\Gamma'_{\mathrm{sat}}=\left(\Gamma'_{\mathrm{div}}\right)_{\mathrm{sat}}=\left(\Gamma_{\mathrm{div}}\right)_{\mathrm{sat}}=\Gamma_{\mathrm{sat}}.\qedhere\]
\end{proof}

We are now able to prove Proposition \ref{reduction-step}.
\begin{proof}[Proof of Proposition  \ref{reduction-step}]
Let $G$, $K$, and $\Gamma$ be as in Theorem \ref{cor-main} and let $L$ be a finite extension of $K(\Gamma_{\mathrm{sat}})$. We want to show that $G(L)/\Gamma_{\mathrm{sat}}$ is free abelian. 

Let $K'/K$ be a finite extension such that there is an isogeny, defined over $K'$, from $G$ to a product $G'= \mathbb{G}_{m,K'}^t \times A$, where $t \in \mathbb{Z}_{\geq 0}$, $A$ is an abelian variety over $K'$, and all geometric endomorphisms of $A$ are defined over $K'$. We fix an isogeny $\varphi: G \to G'$, defined over $K'$. We denote by $\mathrm{Ker}\varphi$ its kernel and set $d = |\mathrm{Ker}\varphi|$. By Lemma \ref{fact-isogeny-degree}, there is an isogeny $\psi: G'\rightarrow G$, defined over $K'$, such that the morphisms $\psi\circ \varphi$ and $\varphi\circ \psi$ are the multiplication by $d$ on $G$ and $G'$ respectively.

By Lemma \ref{lem-deduction}, there exist a finite extension $E/K'$ with $E\subset K'(\Gamma_{\mathrm{sat}})$ and a finitely generated subgroup $\Gamma^\circ$ of $G(E)$ such that $\Gamma_{\mathrm{sat}}=\Gamma^\circ_{\mathrm{sat}}$. We set $\Gamma' = \varphi(\Gamma^\circ)$. Then $\Gamma'$ is a finitely generated subgroup of $G'(E)$. We now want to show that $\Gamma'_{\mathrm{sat}} =\varphi(\Gamma_{\mathrm{sat}})$.

This is equivalent to proving that $\varphi(\Gamma^\circ_{\mathrm{sat}})=\varphi(\Gamma^\circ)_{\mathrm{sat}}$. So let $P'\in \varphi(\Gamma^\circ_{\mathrm{sat}})$. Then $P'=\varphi(P)$ for some $P\in G(\overline{K})$ such that \[mP=\varphi_1(\gamma_1)+\cdots+\varphi_r(\gamma_r)\] for some integer $m\geq 1$, $\varphi_i\in \mathrm{End}(G)$, and $\gamma_i\in \Gamma^\circ$ ($i = 1, \hdots, r$).
Set $\varphi_i'=\varphi\circ \varphi_i\circ \psi$. Then $\varphi_i'\in\mathrm{End}(G')$ and 
\[dmP'=\sum_{i=1}^r(\varphi\circ \varphi_i)(d\gamma_i)=\sum_{i=1}^r(\varphi\circ \varphi_i)\left((\psi\circ\varphi)(\gamma_i)\right)=\varphi_1'(\varphi(\gamma_1))+\cdots+\varphi_r'(\varphi(\gamma_r))\]
and hence $P'\in \varphi(\Gamma^\circ)_{\mathrm{sat}}$. On the other hand, let $Q'\in \varphi(\Gamma^\circ)_{\mathrm{sat}}$ and let $Q\in G(\overline{K})$ such that $Q'=\varphi(Q)$ (recall that $\varphi$ is surjective). We want to show that $Q\in \Gamma^\circ_{\mathrm{sat}}$. As $Q'\in \varphi(\Gamma^\circ)_{\mathrm{sat}}$, we have that \[nQ'=\varphi_1'(\varphi(\gamma_1))+\cdots+\varphi_s'(\varphi(\gamma_s))\] for some integer $n\geq 1$, $\varphi_i'\in\mathrm{End}(G')$, and $\gamma_i\in \Gamma^\circ$ ($i = 1, \hdots, s$). We set $\varphi_i=\psi\circ\varphi_i'\circ \varphi\in \mathrm{End}(G)$ for $i = 1, \hdots, s$. Recalling that $\varphi\circ \psi$ is the multiplication by $d$ on $G'$, we deduce that \[dnQ'=\varphi(\varphi_1(\gamma_1)+\cdots+\varphi_s(\gamma_s)).\]
But then
$dnQ=\varphi_1(\gamma_1)+\cdots+\varphi_s(\gamma_s)+T$ with $T\in (\mathrm{Ker} \varphi)(\overline{K})$. In particular $d^2nQ\in \mathrm{End}(G)\cdot \Gamma^\circ$, which proves the sought-for equality.

Since $(\mathrm{Ker} \varphi)(\overline{K}) \subset G_{\mathrm{tor}} \subset \Gamma_{\mathrm{sat}}$, we also have that
\[ \varphi^{-1}(\Gamma'_{\mathrm{sat}}) = \Gamma_{\mathrm{sat}}.\]
Therefore, the isogeny $\varphi$ induces an injective group homomorphism
\[ G(L)/\Gamma_{\mathrm{sat}} \to G(LE)/\Gamma_{\mathrm{sat}} \to G'(LE)/\Gamma'_{\mathrm{sat}}.\]

We now show that $K'(\Gamma_{\mathrm{sat}}) = E(\Gamma'_{\mathrm{sat}})$. Indeed, $E(\Gamma'_{\mathrm{sat}})\subset K'(\Gamma_{\mathrm{sat}})$ by construction and by the fact that $\phi(\Gamma_{\mathrm{sat}}) = \Gamma'_{\mathrm{sat}}$, while the other inclusion follows from the fact that $\psi(\Gamma'_{\mathrm{sat}})=(\psi\circ \varphi)(\Gamma_{\mathrm{sat}})=d \Gamma_{\mathrm{sat}}=\Gamma_{\mathrm{sat}}$ since, by \eqref{Gamma-sat-div}, $\Gamma_{\mathrm{sat}}=\left(\Gamma_{\mathrm{sat}}\right)_{\mathrm{div}}$.
Since subgroups of free abelian groups are free abelian (see for instance \cite[Ch. I, Theorem 7.3]{Lang02}) and since $LE$ is a finite extension of $K'(\Gamma_{\mathrm{sat}}) = E(\Gamma'_{\mathrm{sat}})$, we see that it suffices to prove Theorem \ref{cor-main} for $G'$, $E$, and $\Gamma'$ instead of $G$, $K$, and $\Gamma$.
\end{proof}

\section{Pottmeyer's criterion and a further reduction}\label{reduction-second-step}
The goal of this section is to reduce Theorem \ref{cor-main} to Theorem \ref{lem:arbitraryrank-var-ab}:
\begin{proposition}\label{theorem3-implies-theorem2}
    Theorem \ref{lem:arbitraryrank-var-ab} implies Theorem \ref{cor-main}.
\end{proposition}
To prove this result, a key ingredient is the following lemma from \cite{Pot}, which provides a criterion for the necessary condition, identified by Pottmeyer, for the generalisation of R\'emond's conjecture to products of powers of $\mathbb{G}_m$ and abelian varieties to hold. As it is a central tool for completing our proof, we recall its proof here for clarity. A precursor and special case of this lemma already appears as Proposition 2.3 in \cite{GHP15}.

\begin{lemma}[{\cite[Lemma 3.1]{Pot}}]\label{Lukas-with-F} Let $G=\mathbb{G}_{m,E}^t\times A$ where $t\in\mathbb{Z}_{\geq 0}$, $E$ is a number field with a fixed algebraic closure $\overline{E}$, and $A$ is an abelian variety defined over $E$.
Let $\Gamma\subset G(\overline{E})$  be an arbitrary subgroup.
Suppose that there exists a field $E \subset E'\subset E(\Gamma_{\mathrm{sat}})$ such that, for every intermediate field  $E'\subset F \subset E(\Gamma_{\mathrm{sat}})$ with $[F:E']<\infty$, one has that
\begin{enumerate}[(i)]
\item\label{cond-i} $G(F)/(\Gamma_{\mathrm{sat}}\cap G(F))$ is free abelian and 
\item\label{cond-ii} the torsion subgroup of $G(E(\Gamma_{\mathrm{sat}}))/(G(F)+\Gamma_{\mathrm{sat}})$ has finite exponent.
\end{enumerate}
Then $G(E(\Gamma_{\mathrm{sat}}))/\Gamma_{\mathrm{sat}}$ is free abelian.
\end{lemma}
\begin{proof}
By a result of Pontryagin recalled in \cite[Ch. IV, Theorem 2.3]{EM02} and \cite[Ch. IV, Lemma 2.2 (iv)]{EM02}, since the group $\Lambda=G(E(\Gamma_{\mathrm{sat}}))/\Gamma_{\mathrm{sat}}$ is countable, 
it is enough to prove that every finite subset $S\subset \Lambda$ is contained in a free abelian subgroup $H \subset \Lambda$ such that the quotient $\Lambda/H$ is torsion-free.

We let $S=\{[\alpha_1],\ldots,[\alpha_s]\} \subset G(E(\Gamma_{\mathrm{sat}}))/\Gamma_{\mathrm{sat}}$, where 
$\alpha_i\in G(E(\Gamma_{\mathrm{sat}}))$ and $[\alpha_i]$ denotes its class in the quotient.
We also set $F= E'(\alpha_1, \ldots, \alpha_s)$. Then
$F/E'$ is finite, $F\subset E(\Gamma_{\text{sat}})$, and $S$ is a subset of the quotient group
\[ 
\left(G(F)+\Gamma_{\text{sat}}\right)/\Gamma_{\text{sat}} \cong G(F)/\left(\Gamma_{\text{sat}} \cap G(F)\right),
\]
which is free abelian by condition \eqref{cond-i}.

By condition \eqref{cond-ii}, the torsion subgroup of
\[
(G(E(\Gamma_{\text{sat}}))/\Gamma_{\text{sat}})/((G(F)+\Gamma_{\text{sat}})/\Gamma_{\text{sat}}) \cong G(E(\Gamma_{\text{sat}}))/(G(F)+\Gamma_{\text{sat}})
\]
has finite exponent, say $m\in \mathbb{N}$. 
We now set
\[
H= \{[\alpha] \in G(E(\Gamma_{\text{sat}}))/\Gamma_{\text{sat}} \mid m \cdot [\alpha] \in (G(F)+\Gamma_{\text{sat}})/\Gamma_{\text{sat}}\}.
\]
The endomorphism $[x] \mapsto [mx]$ of $G(E(\Gamma_{\text{sat}}))/\Gamma_{\text{sat}}$ is injective by \eqref{Gamma-sat-div} and it induces an isomorphism between $H$ and a subgroup of $(G(F)+\Gamma_{\text{sat}})/\Gamma_{\text{sat}}$. Since $(G(F)+\Gamma_{\text{sat}})/\Gamma_{\text{sat}}$ is free abelian, so is $H$ (see \cite[Ch. I, Theorem 7.3]{Lang02}). Moreover, by construction, $S \subset H$. We are now left with proving that the quotient of $G(E(\Gamma_{\text{sat}}))/\Gamma_{\text{sat}}$ by $H$ is torsion-free. Indeed, if $[\alpha] \in G(E(\Gamma_{\text{sat}}))/\Gamma_{\text{sat}}$ is an element whose class modulo $H$ is torsion, then also the class of $\alpha$ in $(G(E(\Gamma_{\text{sat}}))/\Gamma_{\text{sat}})/((G(F)+\Gamma_{\text{sat}})/\Gamma_{\text{sat}})$ is torsion and so $m \cdot [\alpha] \in (G(F)+\Gamma_{\text{sat}})/\Gamma_{\text{sat}}$ and $[\alpha] \in H$ by definition of $H$. This concludes the proof.
\end{proof}

We shall also need the following result proved in \cite[Proposition 3.3]{Pot}.
\begin{proposition}\label{Pot-prop3.3}
 Let $G=\mathbb{G}_{m,E}^t\times A$ where $t\in\mathbb{Z}_{\geq 0}$, $E$ is a number field with a fixed algebraic closure $\overline{E}$, and $A$ is an abelian variety defined over $E$.
Let $F$ be a subfield of $\overline{E}$ with $F/E$ finite and let $\Gamma$ be a finite rank subgroup of $G(\overline{E})$. For any field $E \subset F' \subset F(G_{\mathrm{tor}})$, the group $G(F')/ (\Gamma_{\mathrm{sat}}\cap G(F'))$ is free abelian.
\end{proposition}
We are now able to prove Proposition \ref{theorem3-implies-theorem2}.
\begin{proof}[Proof of Proposition \ref{theorem3-implies-theorem2}]
Suppose that Theorem \ref{lem:arbitraryrank-var-ab} holds true. By Proposition \ref{reduction-step}, it is enough to prove that Theorem \ref{cor-main} holds true for any semiabelian variety of the form $G=\mathbb{G}_{m,K}^{t} \times A$ and any finitely generated subgroup $\Gamma$ of $G(K)$, where $K$ is a number field, $t \in \mathbb{Z}_{\geq 0}$, and $A$ is an abelian variety defined over $K$ having all geometric endomorphisms defined over $K$.

We fix such a semiabelian variety $G$ and finitely generated group $\Gamma\subset G(K)$. We let $L/K(\Gamma_{\mathrm{sat}})$ be a finite extension and we want to show that $G(L)/\Gamma_{\mathrm{sat}}$ is free abelian.

Let $E/K$ be a finite extension such that $L=E(\Gamma_{\mathrm{sat}})$. Clearly, $G$ is defined over $E$ and $\Gamma$ is a finite rank subgroup of $G(\overline{K})$.   
We claim that the hypotheses \eqref{cond-i} and \eqref{cond-ii} of Lemma \ref{Lukas-with-F} are satisfied by the field $E'=E$.

To prove this, we fix an intermediate field $E\subset F\subset E(\Gamma_{\mathrm{sat}})$ such that $[F:E] < \infty$.

Then, the validity of hypothesis \eqref{cond-i} follows from Proposition \ref{Pot-prop3.3} applied to $F'=F$, also thanks to the trivial observation that $F\subset F(G_{\mathrm{tor}})$. 

To verify hypothesis \eqref{cond-ii}, we must show that the torsion subgroup of $G(E(\Gamma_{\mathrm{sat}}))/(G(F) + \Gamma_{\mathrm{sat}})$ has finite exponent. 

Suppose therefore that $P \in G(E(\Gamma_{\mathrm{sat}}))$ is such that $nP \in G(F) + \Gamma_{\mathrm{sat}}$ for some integer $n\geq 1$. 
Since, by \eqref{Gamma-sat-div} in Remark \ref{rmk:sat-div-sat}, $\Gamma_{\text{sat}}=(\Gamma_{\mathrm{sat}})_{\mathrm{div}}$, there exists $Q \in \Gamma_{\mathrm{sat}}$ such that $n(P-Q) \in G(F)$ and hence $P-Q \in G(E(\Gamma_{\mathrm{sat}})) \cap G(F)_{\mathrm{div}} \subset G(F(\Gamma_{\mathrm{sat}})) \cap G(F)_{\mathrm{div}}$. 

Then, Theorem \ref{lem:arbitraryrank-var-ab} applies (over the ground field $F$) to show that there exists a positive integer $d= d(G, \Gamma, F)$ such that $d(P-Q)\in G(F)+\Gamma_{\mathrm{sat}}$. As $Q\in \Gamma_{\text{sat}}$, so is $dQ$. Hence $dP\in G(F)+\Gamma_{\text{sat}}$ and $d$ is a bound for the exponent of the torsion subgroup of $G(E(\Gamma_{\mathrm{sat}}))/(G(F) + \Gamma_{\mathrm{sat}})$. 

So the hypotheses of Lemma \ref{Lukas-with-F} are satisfied and $G(E(\Gamma_{\mathrm{sat}}))/\Gamma_{\mathrm{sat}}=G(L)/\Gamma_{\mathrm{sat}}$ is free abelian.
\end{proof}

\section{Splitting of the endomorphism ring and the saturated closure}\label{splitting-prelim}
\subsection{Splitting of the endomorphism ring} 
In this section we collect some classical results on the endomorphism rings of split semiabelian varieties. 
For a commutative algebraic group $G$ defined over a field $K$, we denote by $\mathrm{End}_K(G)$ the ring of endomorphisms of $G$ that are defined over $K$. We recall that we use the additive notation even when referring to the group law on tori.
\begin{lemma}\label{dec-end}
    Let \[G=T \times A,\] where $T=\mathbb{G}_{m,K}^{t}$, $t \in \mathbb{Z}_{\geq 0}$, and $A$ is an abelian variety defined over a field $K$ of characteristic $0$. Then:
    \begin{enumerate}[(a)]
    \item\label{split-endom-part-a} There is a canonical isomorphism 
    \[f: \mathrm{End}_K(G)\rightarrow  \mathrm{End}_K(T)\times \mathrm{End}_K(A)\]
sending $\varphi\in \mathrm{End}_K(G)$ to $f(\varphi)=(\varphi_T,\varphi_A)$ where, for $H \in \{T,A\}$, $\varphi_H \in \mathrm{End}_K(H)$ is defined as $\varphi_H=\pi_H\circ\varphi\circ\iota_H$ for $\iota_H: H \rightarrow G$ the embedding of $H$ as the corresponding factor of $G$ in the decomposition of $G$ and $\pi_H:G\rightarrow H$ the projection onto the corresponding factor of $G$.
\item\label{split-endom-part-b} There is an isomorphism
\[\Phi_0: \mathrm{Mat}_{t\times t}(\mathbb{Z})\rightarrow \mathrm{End}_K(T),\]
sending a matrix $A=(a_{r,s})_{r,s}$ to the endomorphism $\varphi_A\in \mathrm{End}_K(T)$ given by
\[\varphi_A(x_1,\ldots,x_{t})=\left(\sum_{s=1}^{t} {a_{1,s}x_s}, \ldots, \sum_{s=1}^{t} {a_{t,s}x_s}\right).\] 
    \end{enumerate}
\end{lemma}
\begin{proof}
Since there is no non-trivial homomorphism from $T$ to $A$ or vice versa, one can check that the map $f$ defined in part (a) of the lemma is an isomorphism between $\mathrm{End}_K(G)$ and $\mathrm{End}_K(T) \times \mathrm{End}_K(A)$.

For part \eqref{split-endom-part-b}, one can check that the inverse of $\Phi_0$ is defined by sending an endomorphism $\varphi$ to the matrix $(a_{r,s})_{r,s}$, where
\[ a_{r,s}(x) = (p_r \circ \varphi)(0,\hdots,0,\underbrace{x}_{\text{position } s},0,\hdots,0),\]
$p_r: T = \mathbb{G}_{m,K}^{t} \to \mathbb{G}_{m,K}$ denotes the projection to the $r$-th factor, and we identify $\mathrm{End}_K(\mathbb{G}_{m,K})$ with $\mathbb{Z}$.
\end{proof}

\subsection{Splitting of the saturated closure}

\begin{lemma}\label{lem-decomp-Gamma-sat}
    Let \[G=\mathbb{G}_{m,K}^{t}\times A\] where $K$ is a number field, $t\in \mathbb{Z}_{\geq 0}$, and $A$ is an abelian variety defined over $K$.
    
    If $\Gamma\subset G(K)$ is a finitely generated subgroup, then
\[ \Gamma_{\mathrm{sat}} = (\Gamma_T)_{\mathrm{sat}} \times (\Gamma_A)_{\mathrm{sat}} = (\Gamma_{\mathbb{G}_m})_{\mathrm{sat}}^{t}\times (\Gamma_A)_{\mathrm{sat}},\]
where, for $H \in \{\mathbb{G}_m,T,A\}$, $\Gamma_H$ is a finitely generated subgroup of $H(K)$.
\end{lemma}
\begin{proof}
We keep using the additive notation even when referring to the group law on $\mathbb{G}_{m,K}$.

By Lemma \ref{dec-end}, we have
 a canonical isomorphism  \begin{equation}\label{dec-end-in-proof}\mathrm{End}(G)\simeq \mathrm{End}(T)\times \mathrm{End}(A).\end{equation}
For $H \in \{T,A\}$, let $\iota_H:H\rightarrow G$ denote the embedding of $H$ as a factor of $G$ in the decomposition of $G$ and let $\pi_H:G\rightarrow H$ denote the projection onto the corresponding factor of $G$. We set \[\Gamma_H=\pi_H(\Gamma).\]

The proof now consists of two steps:
\begin{enumerate} [(i)]
    \item\label{step1} first, we show that $\Gamma_{\mathrm{sat}}=(\Gamma_T)_{\mathrm{sat}}\times(\Gamma_A)_{\mathrm{sat}}$;
    \item\label{step2} then, we conclude by proving that $(\Gamma_T)_{\mathrm{sat}}=(\Gamma_{\mathbb{G}_m})_{\mathrm{sat}}^{t}$ where $\Gamma_{\mathbb{G}_m}$ is a finitely generated subgroup of $\mathbb{G}_m(K)$.
\end{enumerate}

To prove \eqref{step1}, 
we first want to show that
\begin{equation}\label{dec-EndA-Gamma}\mathrm{End}(G)\cdot \Gamma=\mathrm{End}(T)\cdot \Gamma_T\times \mathrm{End}(A)\cdot \Gamma_A.\end{equation}
The left-to-right inclusion follows from Lemma \ref{dec-end} and the fact that
$\Gamma \subset \Gamma_T \times\Gamma_A$.
On the other hand, if \[P=(P_T,P_A)\in \mathrm{End}(T)\cdot \Gamma_T\times \mathrm{End}(A)\cdot \Gamma_A,\] then, for each $H \in \{T,A\}$, we have 
$P_H=\alpha_{H,1}(\gamma_{H,1})+\cdots+\alpha_{H,m_H}(\gamma_{H,m_H})$
for some integer $m_H\geq 1$, $\gamma_{H,1},\ldots,\gamma_{H,m_H}\in \Gamma_H$, and $\alpha_{H,1},\ldots,\alpha_{H,m_H}\in \mathrm{End}(H)$.

Let now $\gamma_{H,j}'$ be a preimage of $\gamma_{H,j}$ in $\Gamma$. Then 
\[\iota_H(P_H)=\sum_{j=1}^{m_H}{(\iota_H \circ \alpha_{H,j}\circ \pi_{H})(\gamma'_{H,j})} \in \mathrm{End}(G)\cdot \Gamma\]
and finally we get that $P = \iota_T(P_T) + \iota_A(P_A) \in \mathrm{End}(G)\cdot \Gamma$, proving \eqref{dec-EndA-Gamma}.

We can now conclude the proof of \eqref{step1}. Indeed, if $P\in \Gamma_{\mathrm{sat}}$, there exists an  integer $n\geq 1$ such that $ nP\in \mathrm{End}(G)\cdot \Gamma$. But then, by \eqref{dec-EndA-Gamma}, $P=(P_T,P_A)$ with  $P_H\in H(\overline{K})$ for $H \in \{T,A\}$ and $ nP_H\in \mathrm{End}(H)\cdot \Gamma_H$. Hence $P_H\in (\Gamma_H)_{\mathrm{sat}}$ for $H \in \{T,A\}$.

On the other hand, if $P=(P_T,P_A)$ with $P_H\in (\Gamma_H)_{\mathrm{sat}}$ for $H \in \{T,A\}$, then there are positive integers $n_T,n_{A}$ such that $n_H P_H \in \mathrm{End}(H)\cdot \Gamma_H$ for $H \in \{T,A\}$. Taking $n=\mathrm{lcm}(n_T,n_A)$, we have, by \eqref{dec-EndA-Gamma}, that $nP\in \mathrm{End}(G)\cdot \Gamma$, so $P\in \Gamma_{\mathrm{sat}}$.

Finally, to prove \eqref{step2}, we set \[\Gamma_{T,j}=\beta_{j}(\Gamma_T)\] where $\beta_{j}: T\rightarrow \mathbb{G}_{m,K}$ is the projection onto the $j$-th factor of $T = \mathbb{G}_{m,K}^{t}$.

Let $\Gamma_{\mathbb{G}_m}$ denote the subgroup of $\mathbb{G}_m(K)$ generated by all $\Gamma_{T,j}$'s.

We claim that 
\begin{equation}\label{Gamma-i-sat-ab}(\Gamma_T)_{\mathrm{sat}}=(\Gamma_{\mathbb{G}_m})_{\mathrm{sat}}^{t}.\end{equation}

For the left-to-right inclusion, let $P=(P_1,\ldots,P_{t})\in (\Gamma_T)_{\mathrm{sat}}$. Then, by Lemma \ref{dec-end}, one has
$nP={\alpha}_{1}({\gamma}_{1})+\cdots+{\alpha}_{m}({\gamma}_{m})$ for some integers $n, m\geq 1$, matrices ${\alpha}_{1},\ldots,{\alpha}_{m}\in \mathrm{Mat}_{t\times t}(\mathbb{Z})$, and ${\gamma}_{1},\ldots,{\gamma}_{m}\in \Gamma_T$, where, for a fixed index $s \in \{1,\hdots,m\}$,  the action of the matrix $\alpha_s$ on the point $\gamma_s$ is given as follows: we have $\alpha_s=(a_{h,r})_{1\leq h,r\leq t}$, for some $a_{h,r}\in\mathbb{Z}$, and ${\gamma}_{s}=({\gamma}_{s,1},\ldots, {\gamma}_{s,t})$ where  $ {\gamma}_{s,j}=\beta_{j}({\gamma}_{s})\in \Gamma_{\mathbb{G}_m}$. Then $\alpha_s(\gamma_s)$ is the row vector with $t$ components whose $d$-th coordinate is $a_{d,1}\gamma_{s,1}+\cdots+a_{d,t}\gamma_{s,t}$. Thus, for every $1\leq j\leq t$,
 \[nP_j\in \sum_{i=1}^m\sum_{r=1}^{t}\mathbb{Z} \cdot {\gamma}_{i,r}\subset  \Gamma_{\mathbb{G}_m},\]
 so $P_j\in (\Gamma_{\mathbb{G}_m})_{\mathrm{sat}}$ for every $1\leq j\leq t$. 
 
Conversely, suppose
that $P=(P_1,\ldots,P_{t})$ with $P_r\in (\Gamma_{\mathbb{G}_m})_{\mathrm{sat}}$ for every $1\leq r\leq t$. 
Then, by Lemma \ref{dec-end}, we can find an integer $n\geq 1$ such that, for all $r$'s, we have $nP_r\in \mathbb{Z} \cdot \Gamma_{\mathbb{G}_m}= \Gamma_{\mathbb{G}_m}$. 
In order to conclude, we are left with showing that, for every $1\leq r\leq t$, we have
\begin{equation}\label{nP_r-in-End-B_i-Gamma_i}(0,\ldots,0,nP_r,0,\ldots,0)\in \mathrm{End}(T) \cdot \Gamma_{T}.\end{equation}
Fix now $r\in \{1,\ldots, t\}$.
As $\Gamma_{\mathbb{G}_m}$ is generated by the groups $\Gamma_{T,j}$ for $1\leq j\leq t$, we have that \[nP_r=\sum_{j=1}^{t} \gamma_{j}\]
for some $\gamma_{j}\in \Gamma_{T,j}$.

Let $\gamma_{j}' \in \Gamma_T$ be such that $\beta_j(\gamma'_{j})=\gamma_{j}$. Let also $\sigma_{r,j}\in \mathrm{End}(T)$ denote the endomorphism that sends a point $(Q_1, \hdots, Q_t)$ of $T=\mathbb{G}_{m,K}^{t}$ to the point $(Q_1', \hdots, Q_t')$, where
\[Q_i' = \begin{cases}
    Q_j \quad &\text{if } i = r,\\
    0 \quad &\text{otherwise.}
\end{cases}\]

Then $(0,\ldots,0,nP_r,0,\ldots,0)=\sum_{j=1}^{t} \sigma_{r,j}(\gamma'_{j})$, proving \eqref{nP_r-in-End-B_i-Gamma_i}. 
\end{proof}

\section{More on the saturated closure of finite rank subgroups of the multiplicative group}\label{prel-gen-Gamma-sat}
Let $K$ be a number field. In this section, we prove a lemma on generators of the saturated closure of a finite rank subgroup $\Gamma$ of $\mathbb{G}_{m,K}(K)$.
From now on, we use the convention that, if $n>m$, the set of points $\{P_n,\ldots,P_m\}$ is the empty set and $\langle P_n,\hdots,P_{m}\rangle=\{0\}$. We also use the convention that $(P_n,\ldots,P_m) = 0$ if the $P_i$'s belong to $\mathbb{G}_{m,K}(K)$ and $n > m$.

The result that we are going to prove essentially states that for any fixed collection of points in $\mathbb{G}_{m,K}(K)$, one can always find another set of $s$ linearly independent points that describes not only $\Gamma_{\mathrm{sat}}$, but also the saturated closure of the subgroup generated by $\Gamma$ and the points we fixed. Additionally, 
the group of $K$-points of the saturated closure of the subgroup of $\mathbb{G}_{m,K}^s(K)$ generated by such an $s$-tuple of points is of ``controlled size". We use the convention that $\mathbb{G}^0_{m,K}$ is the trivial algebraic group.

\begin{lemma}\label{lem:independent-generators-tor}
    Let $K$ be a number field and let $\Gamma \subset \mathbb{G}_{m,K}(K)$ be a subgroup of finite rank $r$.
    
    Let $B = B(K)$ be the number of roots of unity in $K$. Then, for every  $k \in \mathbb{Z}_{\geq 0}$ and 
    \[ Q = (Q_1,\hdots,Q_k) \in \mathbb{G}_{m,K}^k(K),\]
    there exist $s \in \{r,\hdots,r+k\}$ and $\mathbb{Z}$-linearly independent points $P_1,\hdots,P_s \in \mathbb{G}_{m,K}(K)$ such that
    \begin{enumerate}[(i)]
    \item\label{lemma13-i} $\Gamma_{\mathrm{sat}} = \langle P_1,\hdots,P_r\rangle_{\mathrm{div}}$,
    \item\label{lemma13-ii}  $\langle \Gamma, Q_1,\hdots, Q_k\rangle_{\mathrm{sat}} = \langle P_1,\hdots,P_{s}\rangle_{\mathrm{div}}$,
    \item\label{lemma13-iii}   the exponent of the quotient group
    \[ \left(\mathbb{G}^s_{m,K}(K) \cap (\mathrm{End}(\mathbb{G}^s_{m,K}) \cdot (P_1,\hdots,P_s))_{\mathrm{div}}\right) / \mathrm{End}(\mathbb{G}^s_{m,K}) \cdot (P_1,\hdots,P_s)\]
    divides $B$.
    \end{enumerate}
\end{lemma}

\begin{proof}   
    We write $\mathbb{G}_m$ instead of $\mathbb{G}_{m,K}$ for simplicity. Throughout this proof, we will use without further note that, since $\mathrm{End}(\mathbb{G}_m)$ is isomorphic to $\mathbb{Z}$, the equality $\Delta_{\mathrm{div}} = \Delta_{\mathrm{sat}}$ holds for every subgroup $\Delta$ of {$\mathbb{G}_m({K})$}.
    Also, for a finite set of places $S$ of $K$, we denote by $\mathcal{O}_{K,S}^{\times}$ the group of units of the ring of $S$-integers of $K$.
    For ease of notation, we generally assume that $r \geq 1$, but the reader can check themselves that our proof also goes through literally for $r = 0$ if one uses the convention recalled at the beginning of the section.
    
    Since $\Gamma$ is of finite rank, there exists a finite set $S_0$ of places of $K$ such that $\Gamma \subset \mathcal{O}_{K,S_0}^\times$. It follows that
    \[ \Gamma_{\mathrm{div}} \cap \mathbb{G}_m(K) \subset  \mathcal{O}_{K,S_0}^\times\]
    is a subgroup of a finitely generated abelian group and hence finitely generated. Furthermore, the rank of $\Gamma_{\mathrm{div}} \cap \mathbb{G}_m(K)$ is equal to the rank of $\Gamma$, which is equal to $r$. It follows that
    \[ \Gamma_{\mathrm{div}} \cap \mathbb{G}_m(K) \simeq \mathbb{G}_m(K)_{\mathrm{tor}} \times \mathbb{Z}^r,\]
    where $\mathbb{G}_m(K)_{\mathrm{tor}}$ is as defined in \eqref{G_K_tor}.
    Hence, there are $\mathbb{Z}$-linearly independent points $P_1, \hdots, P_r \in \mathbb{G}_m(K)$ such that \eqref{lemma13-i} holds, that is
    \[ \Gamma_{\mathrm{sat}} = \langle P_1,\hdots,P_r \rangle_{\mathrm{div}},\]
    and the exponent of 
    \[(\langle P_1,\hdots,P_r \rangle_{\mathrm{div}} \cap \mathbb{G}_m(K))/\langle P_1,\hdots,P_r \rangle\]
    is equal to the cardinality of $\mathbb{G}_m(K)_{\mathrm{tor}}$.
    
 Because of the isomorphism $ \mathrm{Mat}_{s \times s}(\mathbb{Z})\simeq \mathrm{End}(\mathbb{G}^s_m)$ of Lemma \ref{dec-end}\eqref{split-endom-part-b}, to complete the proof of the lemma,
    it suffices to find $P_{r+1},\hdots,P_s \in \mathbb{G}_m(K)$ such that $s \leq r+k$, $P_1,\hdots,P_s$ are $\mathbb{Z}$-linearly independent, and for
    \begin{equation}\label{defn-gamma-prime}
        \Gamma' = \langle P_1,\hdots,P_s \rangle,
    \end{equation}
    we have that 
    \[\Gamma'_{\mathrm{div}} = \langle \Gamma,Q_1,\hdots,Q_k\rangle_{\mathrm{sat}}\]
    and that the exponent of
    \[ (\Gamma'_{\mathrm{div}} \cap \mathbb{G}_m(K))/\Gamma'\]
    divides $B$.

        We set 
    \[ \Gamma_Q = \langle \Gamma,Q_1,\hdots,Q_k\rangle_{\mathrm{sat}} \cap \mathbb{G}_m(K).\]
    
   We first show that the abelian group $\Gamma_Q$ is finitely generated. Indeed, there is a finite set $S \supset S_0$ of places of $K$ such that $Q_i \in \mathcal{O}_{K,S}^\times$ for all $i$. It then follows that $\Gamma_Q \subset \mathcal{O}_{K,S}^\times$ and so $\Gamma_Q$ is finitely generated. It follows that also the quotient group
     \[ \Gamma_Q' := \Gamma_Q/(\langle P_1,\hdots,P_r \rangle_{\mathrm{div}} \cap \mathbb{G}_m(K))\]
     is finitely generated. If some positive multiple of a point in $\mathbb{G}_m(K)$ lies in $\langle P_1,\hdots,P_r \rangle_{\mathrm{div}}$, then already the point itself must have belonged to that group. It follows that $\Gamma_Q'$ is torsion-free and therefore free of finite rank. Let $Q_i'$ denote the class of $Q_i$ in $\Gamma_Q'$ for $i = 1, \hdots, k$, let $P' \in \Gamma_Q'$ be arbitrary, and fix a preimage $P$ of $P'$ in $\Gamma_Q$. Then there exists some integer $n\geq 1$ such that
     \[ nP \in \langle \Gamma, Q_1, \hdots, Q_k \rangle.\]
     But $\Gamma \subset \langle P_1, \hdots, P_r\rangle_{\mathrm{div}} \cap \mathbb{G}_m(K)$ and so
     \[ nP' \in \langle Q_1', \hdots, Q_k'\rangle.\]
    Hence, the rank of $\Gamma_Q'$ is at most $k$ and we can choose $P_{r+1}, \hdots, P_s \in \Gamma_Q$ such that $s \in \{r,\hdots,r+k\}$ and the images of $P_{r+1}, \hdots, P_s$ under the quotient homomorphism are a $\mathbb{Z}$-basis of $\Gamma_Q'$.

     It is clear that $P_1,\hdots,P_s$ are $\mathbb{Z}$-linearly independent. For $\Gamma'$ defined as in \eqref{defn-gamma-prime}, it is clear that $\Gamma'\subset \Gamma_Q$ and so $\Gamma'_{\mathrm{div}} \subset (\Gamma_Q)_{\mathrm{div}}$. Moreover, if $P \in (\Gamma_Q)_{\mathrm{div}}$, then there exists $n \geq 1$ such that $nP \in \Gamma_Q$. By our choice of $P_{r+1}, \hdots, P_s$, there exist $a_{r+1},\hdots,a_s \in \mathbb{Z}$ such that
     \[ nP - a_{r+1}P_{r+1} - \cdots - a_sP_s \in \langle P_1, \hdots, P_r \rangle_{\mathrm{div}}.\]
     It follows that 
     \[ P \in \langle P_1, \hdots, P_s \rangle_{\mathrm{div}} = \Gamma'_{\mathrm{div}}.\]
     We conclude that
    \[ \Gamma'_{\mathrm{div}} = (\Gamma_Q)_{\mathrm{div}} = \langle \Gamma,Q_1,\hdots,Q_k\rangle_{\mathrm{sat}}\]
     and so
     \[ \Gamma'_{\mathrm{div}} \cap \mathbb{G}_m(K) = \Gamma_Q.\]
     
     Finally, consider the canonical homomorphism
     \[ \langle P_1,\hdots,P_r \rangle_{\mathrm{div}} \cap \mathbb{G}_m(K) \to \Gamma_Q/\Gamma'.\]
    By the choice of $P_{r+1}, \hdots, P_s$, we have that
    \[  (\langle P_1,\hdots,P_r \rangle_{\mathrm{div}} \cap \mathbb{G}_m(K)) + \Gamma' = \Gamma_Q\]
    and so this homomorphism is surjective. It also factors through
    \[ (\langle P_1,\hdots,P_r \rangle_{\mathrm{div}} \cap \mathbb{G}_m(K))/\langle P_1,\hdots,P_r \rangle.\]
    Therefore, the exponent of $\Gamma_Q/\Gamma' = (\Gamma'_{\mathrm{div}} \cap \mathbb{G}_m(K))/\Gamma'$ divides the exponent of 
    \[(\langle P_1,\hdots,P_r \rangle_{\mathrm{div}} \cap \mathbb{G}_m(K))/\langle P_1,\hdots,P_r \rangle,\]
    which is equal to $B$. This completes the proof of the lemma.
\end{proof}

\section{A result from Kummer theory}\label{sec-Kummer}
Let $G$ be the product of an abelian variety $A$, defined over a number field $K$ with a fixed algebraic closure $\overline{K}$, and a power $T$ of $\mathbb{G}_{m,K}$. Let $P\in G(K)$ and,
for a positive integer $n$, let $P_n\in G(\overline{K})$ denote any point such that $nP_n=P$. Recall that the Galois group $\mathrm{Gal}(\overline{K}/K)$ acts on $G(\overline{K})$ coordinatewise, where one can define the action on $A(\overline{K})$ using any projective embedding of $A$ that is defined over $K$.

The \emph{Kummer map associated with $P$ and $n$} is the map
\[\rho_{P,n}: \mathrm{Gal}(\overline{K}/K(G[n]))\rightarrow G[n]\]
 defined by $\rho_{P,n}(\sigma)=\sigma(P_n)-P_n$.

This is a well-defined group homomorphism that does not depend on the choice of $P_n$. 
The goal of Kummer theory is to study the image of $\rho_{P,n}$. 

The following result is a key tool for proving our main theorem. Its proof, as presented in Appendix \ref{sec-proof-GR-main} of this article, was given, in private communication to the authors, by Ga\"el R\'emond. 
\begin{theorem}[R\'emond]\label{main-GR} Let $G = T \times A$ be the product of an abelian variety $A$, defined over a number field $K$, and a power $T$ of $\mathbb{G}_{m,K}$. Recall that $\mathrm{End}_K(G)$ denotes the ring of endomorphisms of $G$ that are defined over $K$.
Then for every $P\in G(K)$, the quotient group \[\left(G(K)\cap \left(\mathrm{End}_K(G) \cdot P\right)_{\mathrm{div}}\right)/\mathrm{End}_K(G) \cdot P\]
is finite. Denoting by $b_P$ its exponent, 
we have that there exists an integer $\Delta\geq 1$, depending only on $G$ and $K$, such that for every $n\geq 1$,
\[\Delta b_P G_P[n]\subset \mathrm{Im}(\rho_{P,n})\subset G_P[n]\]
where
\begin{equation}\label{def-G_p} G_P = \bigcap_{\substack{\varphi \in \mathrm{End}_K(G),\\ \varphi(P) = 0}} \ker \varphi.\end{equation}
\end{theorem}

We note that a statement similar to Theorem \ref{main-GR} is established in \cite[Proposition A.9]{BHP}, where the constant $b_P$ is not explicitly given. However, the precise dependence of $b_P$ on $P$, obtained using R\'emond's proof presented here, plays a crucial role in our proof of Theorem \ref{lem:arbitraryrank-var-ab} as it is essential to getting the uniformity of the constant $d$ when dealing with the toric component of $G$. We also remark that a similar result in the case where the smallest algebraic subgroup of $G$ containing $P$ is connected was already stated, for instance, in \cite[Theorem 1]{bertrand} with a less precise dependence of the constants (see also \cite[Theorem 5.2]{Ber11}).    
As already mentioned, the proof of Theorem \ref{main-GR} is presented in Appendix \ref{sec-proof-GR-main} and relies, amongst other arguments, on some results about Galois cohomology and on Faltings's finiteness theorem for isogeny classes.

\section{Proof of Theorem \ref{lem:arbitraryrank-var-ab}}\label{sec-end-of-proof}
We recall that $\overline{K}$ denotes a fixed algebraic closure of $K$. We denote $\mathbb{G}_{m,K}$ by $\mathbb{G}_m$ to enhance readability. We further recall that we use the additive notation also for the group law on $\mathbb{G}_m$.
Let $P \in G(K(\Gamma_{\mathrm{sat}})) \cap G(K)_{\mathrm{div}}$, then
\[ P = (P_T,P_A) \]
with $P_T = (P_{T,1},\hdots,P_{T,t}) \in \mathbb{G}_m^{t}(K(\Gamma_{\mathrm{sat}})) \cap \mathbb{G}_m^{t}(K)_{\mathrm{div}}$ and $P_A \in A(K(\Gamma_{\mathrm{sat}})) \cap A(K)_{\mathrm{div}}$.

Recall that by Lemma \ref{lem-decomp-Gamma-sat} we have
\begin{equation}\label{eq:decomp-gamma}
\Gamma_{\mathrm{sat}} = (\Gamma_{T})_{\mathrm{sat}} \times (\Gamma_A)_{\mathrm{sat}} = (\Gamma_{\mathbb{G}_m})^{t}_{\mathrm{sat}} \times (\Gamma_A)_{\mathrm{sat}},
\end{equation}
where, for $H \in \{\mathbb{G}_m,T,A\}$, $\Gamma_{H}$ is a finitely generated subgroup of $H(K)$.

Hence, to prove Theorem \ref{lem:arbitraryrank-var-ab},
we have to show that there exists an integer $d=d(G,\Gamma,K)\geq 1$ such that
\begin{equation}\label{goal-torus}
dP_{T}  \in  T(K)+(\Gamma_{T})_{\mathrm{sat}}
\end{equation}
and \begin{equation}\label{goal-abelian}
 dP_A \in A(K)+(\Gamma_A)_{\mathrm{sat}}.\end{equation}

\subsection{Expressing $P$ via division points} \label{subsec:7.3}
We now want to fix points in $\mathbb{G}_m(K)$ and $A(K)$ related to $(\Gamma_{\mathbb{G}_m})_{\mathrm{sat}}$ and $(\Gamma_A)_{\mathrm{sat}}$ and use them to express the coordinates of $P_T$ and the point $P_A$, respectively.

For the toric part, if we denote by $r_T$ the rank of $\Gamma_{\mathbb{G}_m}$, by
 Lemma \ref{lem:independent-generators-tor}, we find an integer  $s_T\in \{r_T,\ldots,r_T+t\}$ and $s_T$ $\mathbb{Z}$-linearly independent points
$\widehat{P} _{T,1}, \hdots, \widehat{P}_{T,r_T}, \widehat{P}_{T,r_T+1}, \hdots,  \widehat{P}_{T,s_T} \in \mathbb{G}_m(K)$
such that
    \begin{equation}\label{eq:generators-gamma-sat-tor}
        (\Gamma_{\mathbb{G}_m})_{\mathrm{sat}} = \langle \widehat{P}_{T,1}, \hdots,  \widehat{P}_{T,r_T}\rangle_{\mathrm{div}},
        \end{equation}

    \begin{equation}\label{eq:generators-division-group-tor}
    \langle \Gamma_{\mathbb{G}_m}, P_{T,1},\hdots,P_{T,t}\rangle_{\mathrm{sat}} = \langle \widehat{P}_{T,1}, \hdots, \widehat{P}_{T,s_T} \rangle_{\mathrm{div}},
    \end{equation}
     and the exponent of the group 
       \begin{equation}\label{eq:quotient-of-division-group-tor}
       \left(\mathbb{G}_m^{s_T}(K) \cap (\mathrm{End}(\mathbb{G}_m^{s_T}) \cdot (\widehat{P}_{T,1},\hdots,\widehat{P}_{T,s_T}))_{\mathrm{div}}\right) / \mathrm{End}(\mathbb{G}_m^{s_T}) \cdot (\widehat{P}_{T,1},\hdots,\widehat{P}_{T,s_T})
       \end{equation}
    is bounded from above by a constant $B = B(K)$.

For the abelian part, we simply choose points
$\widehat{Q} _{A,1}, \hdots, \widehat{Q}_{A,r_A}, \widehat{Q}_{A,r_A+1}, \hdots,  \widehat{Q}_{A,s_A} \in A(K)$
such that $s_A \geq 1$,
    \begin{equation}\label{eq:generators-gamma-sat}
        (\Gamma_A)_{\mathrm{sat}} \cap A(K) = \sum_{i=1}^{r_A}{\mathbb{Z} \cdot \widehat{Q}_{A,i}},
        \end{equation}
    and
    \begin{equation}\label{eq:generators-division-group}
    A(K) = \sum_{i=1}^{s_A}{\mathbb{Z}\cdot \widehat{Q}_{A,i}}.
    \end{equation}
We set
\[    \widehat{P}_A = (\widehat{Q} _{A,1}, \hdots, \widehat{Q}_{A,s_A})\in A^{s_A}(K).\]

The goal of the remaining part of this section is to prove the following claim.
\begin{claim}\label{claim1}
There exists an integer $n \geq 1$ such that, for all $j$, one has
\begin{equation}\label{P_i,j-in-K-delta-with-n}
     P_{T,j} \in \mathbb{G}_m(K(\Delta_{A,n},\Delta_{T,n})) \quad \text{and} \quad P_A \in A(K(\Delta_{A,n},\Delta_{T,n}))
 \end{equation}
 as well as 
 \begin{equation}\label{shape-Pij}
 P_{T,j} = \sum_{i=1}^{s_T}{a_{T,j,i}\widehat{P}_{T,i,n}} + Q_{T,j} \quad \text{and} \quad P_{A} = f_A(\widehat{P}_{A,n}) + Q_{A}\end{equation}
 where, for $m \in \mathbb{N} \backslash \{0\}$, $\Delta_{T,m}$ and $\Delta_{A,m}$ are defined as 
\[ \Delta_{T,m} = \{Q \in \mathbb{G}_m(\overline{K})\mid mQ \in \langle \widehat{P}_{T,1}, \hdots, \widehat{P}_{T,r_T} \rangle\}\]
and
\[ \Delta_{A,m} = \{Q \in A(\overline{K})\mid mQ \in \langle \widehat{Q}_{A,1}, \hdots, \widehat{Q}_{A,r_A} \rangle\}\]
and where, for each $1\leq i\leq s_{T}$, $a_{T,j,i} \in \mathbb{Z}$, $\widehat{P}_{T,i,n} \in \mathbb{G}_m(\overline{K})$ is a point such that \begin{equation}\label{n-division-torus}n\widehat{P}_{T,i,n} = \widehat{P}_{T,i},\end{equation}
$\widehat{P}_{A,n} \in A^{s_A}(\overline{K})$ is a point such that \begin{equation}\label{def-P-hat-A-n}
    n\widehat{P}_{A,n} = \widehat{P}_A,
\end{equation} $f_A: A^{s_A} \to A$ is a homomorphism of abelian varieties which is defined over $K$, and $Q_{T,j} \in \mathbb{G}_m(\overline{K})$ and $Q_A \in A(\overline{K})$ are torsion points of orders dividing $n$.
\end{claim}

\begin{proof}[Proof of Claim \ref{claim1}]
We first prove that there exists an  integer $n'\geq 1$  such that\begin{equation}\label{P_i,j-in-K-delta} P_{T,1}, \hdots, P_{T,t} \in \mathbb{G}_m(K(\Delta_{A,n'},\Delta_{T,n'})) \quad \text{and} \quad P_{A} \in A(K(\Delta_{A,n'},\Delta_{T,n'})).\end{equation}
Indeed, recalling that $P = (P_{T,1}, \hdots, P_{T,t},P_A) \in G(K(\Gamma_{\mathrm{sat}}))$, to prove \eqref{P_i,j-in-K-delta} we are going to show that 
\begin{equation}\label{K(Gamma_sat)}
K(\Gamma_{\mathrm{sat}}) = \prod_{m \in \mathbb{N} \backslash \{0\}}{K(\Delta_{A,m},\Delta_{T,m})}
\end{equation}
where we use the product symbol to denote the compositum of fields.
By \eqref{eq:decomp-gamma}, we have 
 \begin{equation*}
 K(\Gamma_{\mathrm{sat}}) = \prod_{H \in \{A,\mathbb{G}_m\}}{K((\Gamma_H)_{\mathrm{sat}})}.
 \end{equation*}
 It follows from \eqref{eq:generators-gamma-sat-tor} that \[K((\Gamma_{\mathbb{G}_m})_{\mathrm{sat}}) = \prod_{m \in \mathbb{N} \backslash \{0\}}{K(\Delta_{T,m})}.\] 
 Moreover, since every element of $\mathrm{End}(A)$ is defined over $K$, we have that
 \[(\Gamma_A)_{\mathrm{sat}} = ((\Gamma_A)_{\mathrm{sat}} \cap A(K))_{\mathrm{div}}\]
 and so \eqref{eq:generators-gamma-sat} implies that
 \[K((\Gamma_A)_{\mathrm{sat}}) = \prod_{m \in \mathbb{N} \backslash \{0\}}{K(\Delta_{A,m})}.\]
 This establishes \eqref{K(Gamma_sat)}. We then deduce \eqref{P_i,j-in-K-delta} by noting that, given $m_1,\ldots,m_k\in \mathbb{N}\setminus \{0\}$ and setting $m=\mathrm{lcm}(m_1,\ldots,m_k)$, we have \[\prod_{i=1}^k K(\Delta_{A,m_i},\Delta_{T,m_i})\subset K(\Delta_{A,m},\Delta_{T,m}).\] 

Secondly, we now prove that there exists an integer $n''_T \geq 1$ such that, for all $j$, one has
\begin{equation}\label{eq-nP_i,j-sum-s_i}
n''_T P_{T,j} \in \sum_{i = 1}^{s_T}{\mathbb{Z} \cdot \widehat{P}_{T,i}}.\end{equation}
Indeed, this directly follows from \eqref{eq:generators-division-group-tor}.
Furthermore, it follows from \eqref{eq:generators-division-group} that
\[ n''_A P_A \in \sum_{i=1}^{s_A}{\mathbb{Z} \cdot \widehat{Q}_{A,i}}\]
for some integer $n''_A \geq 1$. We set $n'' = \mathrm{lcm}(n''_T,n''_A) \in \mathbb{Z}_{>0}$.

Thus, setting $n=\mathrm{lcm}(n',n'')$, we have that \eqref{P_i,j-in-K-delta-with-n} holds and 
 \begin{equation}\label{eq-nP_i,j-sum-s_i-with-n}
 n P_{T,j} \in \sum_{i = 1}^{s_T}{\mathbb{Z} \cdot \widehat{P}_{T,i}} \quad \text{and} \quad n P_A \in \sum_{i = 1}^{s_A}{\mathbb{Z} \cdot \widehat{Q}_{A,i}} 
 \end{equation}
 for all $j$. In particular, if  
$\widehat{P}_{T,i,n} \in \mathbb{G}_m(\overline{K})$ and $\widehat{P}_{A,n} \in A(\overline{K})$ are such that $n\widehat{P}_{T,i,n} = \widehat{P}_{T,i}$ and $n\widehat{P}_{A,n} = \widehat{P}_A$, by \eqref{eq-nP_i,j-sum-s_i-with-n},  there exist $a_{T,j,i} \in \mathbb{Z}$, a homomorphism of algebraic groups $f_A: A^{s_A} \to A$, sending a point to some fixed $\mathbb{Z}$-linear combination of its coordinates (hence in particular defined over $K$), and torsion points $Q_{T,j}$ and $Q_A$ of orders dividing $n$ such that \eqref{shape-Pij} is also satisfied.
\end{proof}
\subsection{Kummer theory and finale for the abelian part}
\label{subsec:7.4}
The goal of this subsection is to prove \eqref{goal-abelian}. To this aim, we want to apply Theorem \ref{main-GR} to the point
\[ P' = (\widehat{P}_{T},\widehat{P}_A)\]
in $G' = \mathbb{G}_m^{s^{\ast}_T} \times A^{s_A}$, where $s^{\ast}_T = \max\{s_T,1\}$ and
\[\widehat{P}_T = \begin{cases}
    0 &\text{ if } s_T = 0,\\
    (\widehat{P}_{T,1},\hdots,\widehat{P}_{T,s_T}) &\text{ if } s_T \geq 1.
\end{cases}\]
Let $G'_{P'}$, $B_A = (A^{s_A})_{\widehat{P}_A}$, and $B_T = (\mathbb{G}_m^{s^\ast_T})_{\widehat{P}_{T}}$ be as defined in \eqref{def-G_p}.
It then follows from the isomorphism $\mathrm{End}_K(G') \simeq \mathrm{End}_K(\mathbb{G}_m^{s^\ast_T}) \times \mathrm{End}_K(A^{s_A})$ from Lemma \ref{dec-end} that $G'_{P'} = B_T \times B_A$.

If $s_T \geq 1$, we deduce from the $\mathbb{Z}$-linear independence of the $\widehat{P}_{T,i}$ (as constructed in Subsection \ref{subsec:7.3}) and the isomorphism $\mathrm{End}(\mathbb{G}_m^{s_T}) \simeq \mathrm{Mat}_{s_T \times s_T}(\mathbb{Z})$ from Lemma \ref{dec-end} that $B_T = \mathbb{G}_m^{s_T}$. If $s_T = 0$, we use the convention that $\mathbb{G}_m^0 = \{0\} \subset \mathbb{G}_m$ so that in all cases we have 
\begin{equation}\label{ghat-phat}
G'_{P'} = \mathbb{G}_m^{s_T} \times B_A.
\end{equation}

By Theorem \ref{main-GR}, the group
\[ H = \left(G'(K)\cap \left(\mathrm{End}_K(G') \cdot P'\right)_{\mathrm{div}}\right)/\mathrm{End}_K(G') \cdot P'\]
is finite and we denote by $b_{P'}$ its exponent. By Lemma \ref{dec-end}, we have that
\[H \simeq H_T \times H_A\]
where
\[H_T = \left(\mathbb{G}_m^{s^\ast_T}(K)\cap \left(\mathrm{End}_K(\mathbb{G}_m^{s^\ast_T}) \cdot \widehat{P}_T\right)_{\mathrm{div}}\right)/\mathrm{End}_K(\mathbb{G}_m^{s^\ast_T}) \cdot \widehat{P}_T\]
and
\[H_A = \left(A^{s_A}(K)\cap \left(\mathrm{End}_K(A^{s_A}) \cdot \widehat{P}_A\right)_{\mathrm{div}}\right)/\mathrm{End}_K(A^{s_A}) \cdot \widehat{P}_A.\]

By \eqref{eq:quotient-of-division-group-tor}, the exponent of $H_T$ is bounded by a constant that depends only on $K$. Moreover, the exponent of $H_A$ depends only on $\widehat{P}_A$ and $K$, which in turn depend only on $A$, $\Gamma$, and $K$. This means that $b_{P'}$ is bounded by a constant that depends only on $G$, $\Gamma$, and $K$.

Let $n\geq 1$ be the integer from Claim \ref{claim1}. 
It now follows from Theorem \ref{main-GR} and \eqref{ghat-phat} that there exists a positive integer \begin{equation}\label{eq:d_1}d_1 = d_1(s_A,s_T,G,\Gamma,K) = d_1(G,\Gamma,K)\end{equation} such that, for every $R = (R_T,R_A) \in (\mathbb{G}_m^{s_T} \times B_A)[n]$, there exists $\sigma_R \in \mathrm{Gal}(\overline{K}/K(G'[n]))$ such that
\begin{equation} \label{eq:kummer-application-torus} \sigma_R(\widehat{P}_{T,1,n}, \hdots, \widehat{P}_{T,s_T,n}) - (\widehat{P}_{T,1,n}, \hdots, \widehat{P}_{T,s_T,n}) = d_1R_T
\end{equation}
and \begin{equation} \label{eq:kummer-application-abelian} \sigma_R(\widehat{P}_{A,n}) - \widehat{P}_{A,n} = d_1R_A
\end{equation}
where the points $\widehat{P}_{T,i, n}$ and $\widehat{P}_{A,n}$ are defined in  \eqref{n-division-torus}  and \eqref{def-P-hat-A-n}, with the convention that
\[(\widehat{P}_{T,1,n}, \hdots, \widehat{P}_{T,0,n}) = 0.\]
Since $s_A \geq 1$ and $s^{\ast}_T \geq 1$, we have that $K(G'[n]) \supset K(G[n])$ and so
\begin{equation}\label{sigma_R-fixes-G[n]}
    \sigma_R \in \mathrm{Gal}(\overline{K}/K(G[n])).
\end{equation}

We now want to find ``many" points $R$ such that $\sigma_R(P) = P$. For this, we let $\pi: B_A \to A^{r_A}$ denote the projection on the first $r_A$ factors and 
we set $B_A'$ to be the identity component of the algebraic group
\[ B_A \cap \left(\{0\}^{r_A} \times A^{s_A-r_A}\right) = \ker \pi.\]
We also set    
\[ B_T' = \{0\}^{r_T} \times \mathbb{G}_m^{s_T-r_T} \subset \mathbb{G}_m^{s_T}\]
and we choose $R \in (B_T' \times B_A')[n]$. Because of \eqref{eq:kummer-application-torus}, \eqref{eq:kummer-application-abelian}, and \eqref{sigma_R-fixes-G[n]}, this implies that $\sigma_R$ fixes $\Delta_{T,n}$ as well as $\Delta_{A,n}$ (as defined in Claim \ref{claim1}) pointwise and hence, by \eqref{P_i,j-in-K-delta-with-n}, $\sigma_R(P) = P$.

Recall that, by \eqref{shape-Pij},
 \begin{equation}\label{shape-P_A-2}P_{A} = f_A(\widehat{P}_{A,n}) + Q_{A}\end{equation} where $f_A:A^{s_A}\rightarrow A$ is a homomorphism of abelian varieties which is defined over $K$ and $Q_{A}\in A(\overline{K})$ is a torsion point of order dividing $n$.
 
 Therefore, we deduce that
\[0 = \sigma_R(P_A) - P_A = f_A(\sigma_R(\widehat{P}_{A,n})-\widehat{P}_{A,n}) = f_A(d_1R_A).\]
Since $R_A\in B_A'[n]$ was arbitrary, we have that $d_1B_A'[n] \subset \ker f_A$. By \cite[Theorem 5.13]{Milne}, this implies that \begin{equation}\label{d_1(f_A)}
    d_1(f_A)|_{B_A'} = nf_A'
\end{equation} for some homomorphism of algebraic groups $f_A': B_A' \to A$, defined over $K$.

To finish the proof of \eqref{goal-abelian}, we need to prove the following claim.
\begin{claim}\label{claim2}
 There exist an integer $N \geq 1$ and a homomorphism of algebraic groups $\psi: A^{r_A} \to B_A$, defined over $K$ and depending only on $B_A$ and $\pi$, such that, for every point $Q\in B_A(\overline{K})$,
\begin{equation}\label{eq:image-difference}
    NQ - (\psi \circ \pi)(Q) \in B_A'(\overline{K}).
\end{equation}

\end{claim}

\begin{proof}[Proof of Claim \ref{claim2}]
For every $m \in \mathbb{Z}$, we denote by $[m]$ the multiplication-by-$m$ endomorphism on the appropriate algebraic group. 

Let $N_1 \geq 1$ be an integer such that $[N_1](B_A)$ is equal to the identity component $B_A^\circ$ of $B_A$. It follows from Lemma \ref{lem:section-homomorphism} that there exist an integer $N_2 \geq 1$ and a homomorphism of algebraic groups $\widetilde{\psi}: A^{r_A} \to B_A^\circ$, defined over $K$, such that
\[(\pi \circ \widetilde{\psi} \circ \pi)|_{B_A^\circ} = ([N_2] \circ \pi)|_{B_A^\circ} = (\pi \circ [N_2])|_{B_A^\circ}.\]
 Setting $\psi' = N_1\widetilde{\psi}$ and $N' = N_1N_2$, we deduce that
\begin{equation}\label{eq:section}
\pi \circ \psi' \circ \pi = \pi \circ \widetilde{\psi} \circ \pi \circ [N_1] = [N_2] \circ \pi \circ [N_1] = [N'] \circ \pi = \pi \circ [N'].
\end{equation}

Let $N_3 \geq 1$ be an integer such that $[N_3](\ker \pi) = B_A'$. Setting $\psi = N_3\psi'$ and $N = N_3N'$, we deduce from \eqref{eq:section} that
\eqref{eq:image-difference} holds as desired.
\end{proof}

Let $N$ and $\psi$ be as in Claim \ref{claim2}. We deduce from \eqref{shape-P_A-2}, multiplying by $nd_1N$ and adding and subtracting $d_1f_A((\psi \circ \pi)(\widehat{P}_{A}))$, that
\[nd_1NP_A = d_1f_A((\psi \circ \pi)(\widehat{P}_{A}))  + d_1f_A(N\widehat{P}_{A}-(\psi \circ \pi)(\widehat{P}_{A})).\]
Notice that by \eqref{eq:generators-gamma-sat}
\[ d_1f_A((\psi \circ \pi)(\widehat{P}_{A})) = \sum_{i=1}^{r_A}{f_i(\widehat{Q}_{A,i})} \in \langle \widehat{Q}_{A,1}, \hdots, \widehat{Q}_{A,r_A} \rangle_{\mathrm{sat}} \subset ((\Gamma_A)_{\mathrm{sat}})_{\mathrm{sat}} = (\Gamma_A)_{\mathrm{sat}},\]
where $f_i \in \mathrm{End}_K(A)$ sends a point $Q$ to
\[ d_1(f_A \circ \psi)(0, \hdots, 0, \underbrace{Q}_{\text{position } i}, 0, \hdots,0).\]

Moreover $N\widehat{P}_{A}-(\psi \circ \pi)(\widehat{P}_{A}) \in B_A'(\overline{K})$ because of \eqref{eq:image-difference}, so we have that
\[ d_1f_A(N\widehat{P}_{A}-(\psi \circ \pi)(\widehat{P}_{A})) \stackrel{\eqref{d_1(f_A)}}{=} nf_A'(N\widehat{P}_A - (\psi \circ \pi)(\widehat{P}_{A})) \in nA(K).\]
Thus $nd_1NP_A \in (\Gamma_A)_{\mathrm{sat}}+nA(K)$ and so $d_1NP_A \in (\Gamma_A)_{\mathrm{sat}}+A(K)$, which establishes \eqref{goal-abelian} with $d = d_1N$.

\subsection{Finale for the toric part}
In this subsection, we want to conclude the proof also for the toric part by proving \eqref{goal-torus}. If $s_T = 0$, then $P_T$ is a torsion point and \eqref{goal-torus} holds trivially (for any $d \geq 1$). Therefore, we assume from now on that $s_T \geq 1$.

We now establish the following claim, from which we then deduce \eqref{goal-torus}.
\begin{claim}\label{claim3}
For every $r_T+1\leq i\leq s_T$ and every $1 \leq j \leq t$, there exists $a'_{T,j,i}\in \mathbb{Z}$ such that in \eqref{shape-Pij} one has 
\begin{equation}\label{shape-endom}
    a_{T,j,i}=\frac{n}{\gcd(n,d_1)}a'_{T,j,i},
\end{equation}
where $n$ is the integer satisfying \eqref{P_i,j-in-K-delta-with-n} and \eqref{shape-Pij} which was constructed in Subsection \ref{subsec:7.3} and $d_1$ is the integer from \eqref{eq:d_1}.
\end{claim}

\begin{proof}[Proof of Claim \ref{claim3}]
Let $R_T = (0, \hdots, 0, R_{T,r_T+1}, \hdots, R_{T,s_T}) \in B_T'[n]$ be a torsion point, $R_A = 0$, and $R = (R_T,R_A)$ and let $\sigma_{R} \in \mathrm{Gal}(\overline{K}/K(G[n]))$ be a Galois automorphism satisfying \eqref{eq:kummer-application-torus} and \eqref{eq:kummer-application-abelian}.

Fix $j \in \{1,\hdots,t\}$. As in Subsection \ref{subsec:7.4}, we find that $\sigma_R(P) = P$ and so, by \eqref{shape-Pij},
\[0 = \sigma_R(P_{T,j}) - P_{T,j}= \sum_{i=r_T+1}^{s_T}{d_1a_{T,j,i}R_{T,i}}.\]

For $i_0 \in \{r_T+1, \hdots, s_T\}$, we can take $R_{T,i_0}$ of exact order $n$ and $R_{T,i} = 0$ for $i \neq i_0$ to find that $n$ divides $d_1a_{T,j,i_0}$ and so $\frac{n}{\gcd(n,d_1)}$ divides $a_{T,j,i_0}$, which completes the proof of the claim.
\end{proof}
Let now $d=d_1N$ as in Subsection \ref{subsec:7.4}, where $N$ is the integer from Claim \ref{claim2}.
We obtain from \eqref{shape-Pij} together with \eqref{n-division-torus} and Claim \ref{claim3} that for all $j$
\[dP_{T,j} = d\sum_{i=1}^{r_T}{a_{T,j,i}\widehat{P}_{T,i,n}}+\frac{d}{\gcd(n,d_1)}\sum_{i=r_T+1}^{s_T}{a'_{T,j,i}\widehat{P}_{T,i}}+dQ_{T,j} .\]
Notice that \[n\left(d\sum_{i=1}^{r_T}{a_{T,j,i}\widehat{P}_{T,i,n}}\right)=\sum_{i=1}^{r_T}{da_{T,j,i}\widehat{P}_{T,i}}.\]
Hence, by \eqref{eq:generators-gamma-sat-tor},
we have \[d\sum_{i=1}^{r_T}{a_{T,j,i}\widehat{P}_{T,i,n}}\in \left(\sum_{i = 1}^{r_T}{\mathbb{Z}\cdot \widehat{P}_{T,i}}\right)_{\mathrm{div}}
=\left(\Gamma_{\mathbb{G}_m}\right)_{\mathrm{sat}}.\]
Note also that \[\frac{d}{\gcd(n,d_1)}\sum_{i=r_T+1}^{s_T}{a'_{T,j,i}\widehat{P}_{T,i}} \in \mathbb{G}_m(K)\]
and that $dQ_{T,j}\in (\Gamma_{\mathbb{G}_m})_{\mathrm{sat}}$.
Hence $d P_{T,j}\in (\Gamma_{\mathbb{G}_m})_{\mathrm{sat}} + \mathbb{G}_m(K)$, which, by \eqref{eq:decomp-gamma}, establishes \eqref{goal-torus} and concludes the proof of Theorem \ref{lem:arbitraryrank-var-ab}.
\qed

\section*{Acknowledgements}

The authors are grateful to Ga\"el R\'emond for providing the proof of Theorem \ref{main-GR} which is reproduced in Appendix \ref{sec-proof-GR-main}, for his careful reading of a preliminary version of the manuscript, and for his suggestions that helped simplify and improve the exposition significantly, in particular with regard to Subsection \ref{subsec:7.4}. They also thank Philipp Habegger for raising an interesting question and Lukas Pottmeyer for comments on a preliminary version of this work. 

Sara Checcoli received support from the French National Research Agency in the framework of the Investissements d’avenir program (ANR-15-IDEX-02). 
Gabriel Dill thanks the DFG for its support (grant no. EXC-2047/1 - 390685813). He also thanks the Institut Fourier and Sara Checcoli for their hospitality on the occasion of his visit to Grenoble.

This article is based upon work supported by the National Science Foundation under Grant Nos. 1440140 and DMS--1928930 while the authors were in residence at the Mathematical Sciences Research Institute in Berkeley, California, during the Spring 2023 semester. We thank the organisers of the ``Diophantine Geometry'' program for giving us the opportunity to participate in this program
and the MSRI staff for providing an excellent environment for mathematical collaborations.

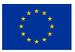
\begin{mdframed}
\begin{tabular*}{0.96\textwidth}{@{\extracolsep{\fill} }cp{0.84\textwidth}}
		% The EU emblem
		\raisebox{-0.7\height}{%
    \begin{tikzpicture}[y=0.80pt, x=0.8pt, yscale=-1, inner sep=0pt, outer sep=0pt, 
				scale=0.12]
				\definecolor{c003399}{RGB}{0,51,153}
				\definecolor{cffcc00}{RGB}{255,204,0}
				\begin{scope}[shift={(0,-872.36218)}]
					\path[shift={(0,872.36218)},fill=c003399,nonzero rule] (0.0000,0.0000) rectangle (270.0000,180.0000);
					\foreach \myshift in 
					{(0,812.36218), (0,932.36218), 
						(60.0,872.36218), (-60.0,872.36218), 
						(30.0,820.36218), (-30.0,820.36218),
						(30.0,924.36218), (-30.0,924.36218),
						(-52.0,842.36218), (52.0,842.36218), 
						(52.0,902.36218), (-52.0,902.36218)}
					\path[shift=\myshift,fill=cffcc00,nonzero rule] (135.0000,80.0000) -- (137.2453,86.9096) -- (144.5106,86.9098) -- (138.6330,91.1804) -- (140.8778,98.0902) -- (135.0000,93.8200) -- (129.1222,98.0902) -- (131.3670,91.1804) -- (125.4894,86.9098) -- (132.7547,86.9096) -- cycle;
				\end{scope}
				%\draw[very thin,dashed] (current bounding box.south west) rectangle               (current bounding box.north east);
			\end{tikzpicture}%
		}
		&
		Gabriel Dill has received funding from the European Research Council (ERC) under the European Union's Horizon 2020 research and innovation programme (grant agreement n$^\circ$ 945714).
	\end{tabular*}
\end{mdframed}

\appendix

\section{Proof of Theorem \ref{main-GR}}\label{sec-proof-GR-main}

The proof of Theorem \ref{main-GR} presented in this appendix is due to Ga\"el R\'emond. We recall that, given a number field $K$, we denote by $\overline{K}$ a fixed algebraic closure of $K$.

We first prove the following lemma:
\begin{lemma}\label{quotient-finite-Remond}
    Let $G$ and $K$ be as in the statement of Theorem  \ref{main-GR}. Let $P\in G(K)$. Then the quotient group 
\[\left(G(K)\cap \left(\mathrm{End}_K(G) \cdot P\right)_{\mathrm{div}}\right)/\mathrm{End}_K(G) \cdot P\]
is finite.
\end{lemma}
\begin{proof}
    Let $G=T\times A$ as in the statement of Theorem \ref{main-GR} with $T=\mathbb{G}_{m,K}^t$ for some $t\in \mathbb{Z}_{\geq 0}$. Set $P=(P_T,P_A)$ where $P_T\in T(K)$ and $P_A\in A(K)$. Let $S$ be a finite set of places of $K$ such that \[P_T\in \left(\mathcal{O}_{K,S}^{\times}\right)^t\] where $\mathcal{O}_{K,S}^{\times}$ is the unit group of the ring of $S$-integers of $K$. We want to show that \[G(K)\cap \left(\mathrm{End}_K(G) \cdot P\right)_{\mathrm{div}}\subset \left(\mathcal{O}_{K,S}^{\times}\right)^t\times A(K).\] Indeed, let $Q=(Q_T,Q_A)\in G(K)\cap \left(\mathrm{End}_K(G) \cdot P\right)_{\mathrm{div}}$. Then, recalling that we use the additive notation also for the group law on $T$ and recalling Lemma \ref{dec-end}(b), there exists $m \in \mathbb{Z}_{> 0}$ such that \[mQ_T \in \mathrm{End}_K(T) \cdot P_T = \mathrm{Mat}_{t \times t}(\mathbb{Z}) \cdot P_T \subset \left(\mathcal{O}_{K,S}^{\times}\right)^t,\] so $Q_T\in \left(\mathcal{O}_{K,S}^{\times}\right)^t$. This proves the claim, as $Q_A\in A(K)$ (since $Q\in G(K)$). 
    
    As, by Dirichlet's unit theorem and the Mordell-Weil theorem, the group $\left(\mathcal{O}_{K,S}^{\times}\right)^t\times A(K)$ is a finitely generated abelian group, so is the group $G(K)\cap \left(\mathrm{End}_K(G) \cdot P\right)_{\mathrm{div}}$. Now let $Q_1,\ldots,Q_r$ be a set of generators for $G(K)\cap \left(\mathrm{End}_K(G) \cdot P\right)_{\mathrm{div}}$. Then the quotient group \[\left(G(K)\cap \left(\mathrm{End}_K(G) \cdot P\right)_{\mathrm{div}}\right)/\mathrm{End}_K(G) \cdot P\] is generated by the classes $[Q_1],\ldots, [Q_r]$, which are all torsion elements (as $Q_i\in \left(\mathrm{End}_K(G) \cdot P\right)_{\mathrm{div}}$), hence it has finite order.
\end{proof}

The proof of Theorem \ref{main-GR} also requires the following two lemmas in Galois cohomology. 

\begin{lemma}\label{lemma1-GR} 
Let $G$ and $K$ be as in the statement of Theorem \ref{main-GR}.
    There exists an integer $\Delta_1\geq 1$, depending only on $G$ and $K$, such that, for all $n\in \mathbb{Z}_{> 0},$
    \[\Delta_1 H^1\left(\mathrm{Gal}(K(G[n])/K),G[n]\right)=0.\]
\end{lemma}

\begin{proof}
Let $G=T\times A$ as in the statement of Theorem \ref{main-GR} so that $T=\mathbb{G}_{m,K}^t$ for some $t\in \mathbb{Z}_{\geq 0}$.
   Let $A^{\vee}$ be the dual of $A$ and fix a polarisation $\phi: A \to A^{\vee}$ defined over $K$. Let $n \in \mathbb{Z}_{>0}$, set $n' = n\deg(\phi)$, and let $m$ be an integer coprime to $n'$.

   We want to construct an element $\sigma$ in the center of $\mathrm{Gal}(K(G[n])/K)$ that acts on both $T[n]$ and $A[n]$ as multiplication by possibly different powers of $m$ whose exponents depend only on $A$ and $K$.
   
   Suppose first that $A[n]$ is non-trivial. By a theorem of Serre \cite[No. 136, Théorème 2']{Ser00} (see also \cite[Théorème 3]{Win}), there exists $\sigma\in \mathrm{Gal}(\overline{K}/K)$ acting on the abelian part $A[n']$ of $G[n']$ as the multiplication by $m^c$, for some integer $c>0$ which depends only on $A$ and $K$. We now show that $\sigma$ acts on the toric part $T[n]=\mu_n^t$ of $G[n]$ as multiplication by $m^{2c}$ (here, as usual, $\mu_n$ denotes the group of $n$-th roots of unity). To prove this, we reason similarly as in the proof of \cite[Lemme 12 (i)]{Hin88}. Let $e_n:A[n]\times A^{\vee}[n]\rightarrow \mu_n$ be the Weil pairing. Let $\zeta\in \mu_n$ and let $(P,P')\in A[n]\times A^{\vee}[n]$ be such that $\zeta=e_n(P,P')$. There exists $P'' \in A[n']$ such that $\phi(P'') = P'$. Then (recalling that we use the additive notation also for $T$)
   \begin{align*}\sigma(\zeta)&=e_n(\sigma(P),\sigma(P'))=e_n(\sigma(P),\phi(\sigma(P''))) = e_n(m^c P, \phi(m^c P'')) \\
   &= e_n(m^c P, m^c P') =m^{2c} e_n(P,P')=m^{2c} \zeta.
    \end{align*}

   If $A[n]$ is trivial, we now show that we can also find $\sigma$ acting on $T[n]$ as the multiplication by $m^{2c}$ for some integer $c > 0$ that depends only on $K$. Indeed, in this case $G[n]=\mu_n^t$ and the action is the usual Galois action on the group of roots of unity.

   In both cases, $\sigma$ induces an element of the group $\mathrm{Gal}(K(G[n])/K)$ which lies in the center because the canonical injective homomorphism $\mathrm{Gal}(K(G[n])/K) \hookrightarrow \mathrm{Aut}(G[n])$ factors through the image of the canonical injective homomorphism $\mathrm{Aut}(T[n]) \times \mathrm{Aut}(A[n]) \hookrightarrow \mathrm{Aut}(G[n])$.

   By Sah's lemma (see for instance \cite[Chapter 8, Lemma 8.1]{Lan83}), we have 
   \[(\sigma-1)H^1\left(\mathrm{Gal}(K(G[n])/K),G[n]\right)=0.\]
   This means that, for every $1$-cocycle $f: \mathrm{Gal}(K(G[n])/K) \to G[n] = T[n] \times A[n]$, the map $(\sigma-1)(f)$, $\tau \mapsto \sigma(f(\tau))-f(\tau)$, is a $1$-coboundary, so there exists $P = (P_T,P_A) \in G[n] = T[n] \times A[n]$ such that
   \[ (\sigma-1)(f)(\tau) = \tau(P)-P\]
   for all $\tau$. Since $\sigma$ acts as multiplication by $m^{2c}$ on $T[n]$ and as multiplication by $m^c$ on $A[n]$, one can check that $(m^{2c}-1)f$ is the $1$-coboundary associated with the point $(P_T,(m^c+1)P_A)$ and so
\[(m^{2c}-1)H^1\left(\mathrm{Gal}(K(G[n])/K),G[n]\right)=0.\]
Let 
\[d=\gcd\{m^{2c}-1\mid \gcd(m,n')=1\}.\]
We also have that
\[d H^1\left(\mathrm{Gal}(K(G[n])/K),G[n]\right)=0.\]
If $x\in \left(\mathbb{Z}/d\mathbb{Z}\right)^{\times}$, then there exists $m\in \mathbb{Z}$ coprime to $n'$ whose image in $\left(\mathbb{Z}/d\mathbb{Z}\right)^{\times}$ equals $x$. 
By definition, $d$ divides $m^{2c}-1$, hence $x^{2c}=1$ or, equivalently, the order of $x$ divides $2c$. Letting $x$ vary, we obtain that the exponent of the group $\left(\mathbb{Z}/d\mathbb{Z}\right)^{\times}$ divides $2c$. As this exponent tends to infinity with $d$, we deduce that $d$ is bounded by a function of $c$ and independently of $n$. Finally, the least common multiple $\Delta_1$ of all such possible $d$'s satisfies the statement of the lemma.  
\end{proof}

\begin{lemma}\label{lemma2-GR} Let $G$, $A$, $T$, and $K$ be as in the statement of Theorem \ref{main-GR}. For every point $P\in G(K)$, we fix a point $P_n\in G(\overline{K})$ such that $n P_n=P$.
Let \[\delta:{G(K)}/{nG(K)}\hookrightarrow H^1\left(\mathrm{Gal}(\overline{K}/K),G[n]\right)
\]
\[[P]\longmapsto [\delta_P]
\]
where $\delta_P(\sigma)=\sigma(P_n)-P_n$. Then $\delta$ is a well-defined group homomorphism.
Let also
\[\mathrm{res}: H^1\left(\mathrm{Gal}(\overline{K}/K),G[n]\right)\rightarrow H^1\left(\mathrm{Gal}(\overline{K}/K(G[n])), G[n]\right)\] 
be the natural restriction. 
Then 
\[H^1\left(\mathrm{Gal}(\overline{K}/K(G[n])), G[n]\right) = \mathrm{Hom}\left(\mathrm{Gal}(\overline{K}/K(G[n])), G[n]\right)\]
and the kernel of the homomorphism $\mathrm{res}\circ \delta$ consists of torsion elements of order dividing the integer $\Delta_1$ from Lemma \ref{lemma1-GR}.

\end{lemma}

\begin{proof}
We first check that the map $\delta$ is well-defined. 
Indeed, if $P \in G(K)$, then, for every $\sigma\in \mathrm{Gal}(\overline{K}/K)$, we have  \[n\delta_P(\sigma)=\sigma(n P_n)-n P_n=\sigma(P)-P=0,\] so $\delta_P(\sigma)\in G[n]$. The map $\delta_P$ is a $1$-cocycle since
\[\delta_P(\sigma \circ \tau) = (\sigma \circ \tau)(P_n) - P_n = \sigma(\tau(P_n) - P_n) + \sigma(P_n) - P_n = \sigma(\delta_P(\tau)) + \delta_P(\sigma)\]
for all $\sigma, \tau \in \mathrm{Gal}(\overline{K}/K)$. Moreover, we can prove that the class $[\delta_P]$ does not depend on the chosen representative in the class $[P]$. Indeed, if $Q\in G(K)$ is such that $Q=P+nQ'$ for some $Q'\in G(K)$ and $Q_n' \in G(\overline{K})$ is a point such that $nQ_n' = Q$, then $Q_n'=P_n+Q'+Q''$ for some $Q''\in G[n]$ and so, for every 
 $\sigma\in \mathrm{Gal}(\overline{K}/K)$, we have $\sigma(Q')=Q'$ and 
 \[\sigma(Q_n') - Q_n'=\sigma(P_n)-P_n+\sigma(Q')-Q'+\sigma(Q'')-Q''=\delta_P(\sigma)+\sigma(Q'')-Q'',\] hence in particular $[\delta_Q]=[\delta_P]\in H^1\left(\mathrm{Gal}(\overline{K}/K),G[n]\right).$ This establishes that $\delta$ is well-defined and that $[\delta_P]$ does not depend on the choice of $P_n$.

We now show that $\delta$ is a homomorphism: if $P, Q \in G(K)$, then
\[ (\delta_P+\delta_Q)(\sigma) = \sigma(P_n+Q_n)-(P_n+Q_n)\]
and so $[\delta_P] + [\delta_Q] = [\delta_{P}+\delta_{Q}] = [\delta_{P+Q}]$ since $n(P_n+Q_n) = P+Q$.

We notice now that  \[\mathrm{Hom}\left(\mathrm{Gal}(\overline{K}/K(G[n])), G[n]\right)=H^1\left(\mathrm{Gal}(\overline{K}/K(G[n])), G[n]\right)\] as the group action on $G[n]$ is trivial and also that, by the inflation-restriction exact sequence \cite[Chapter VII, \S 6, Proposition 4]{Serre_local}, the kernel of $\mathrm{res}$ is $H^1\left(\mathrm{Gal}(K(G[n])/K),G[n]\right)$ (for the definition of inflation, see \cite[Chapter I, \S 5, p. 47]{Neukirch}). Now, suppose that $[P]\in \mathrm{Ker}(\mathrm{res}\circ \delta)$. Then \[[\delta_P]\in H^1\left(\mathrm{Gal}(K(G[n])/K),G[n]\right),\] hence, by Lemma \ref{lemma1-GR},  for all $\sigma \in \mathrm{Gal}(\overline{K}/K)$, we have  $\Delta_1 \delta_P(\sigma)=\sigma(P')-P'$ for some fixed torsion point $P'\in G[n]$. Let $Q=\Delta_1 P_n -P'$. Then $Q\in G(K)$ as \[\sigma(Q)-Q=\sigma(\Delta_1 P_n)-\Delta_1 P_n-\sigma(P')+P'=\Delta_1 \delta_P(\sigma)-(\sigma(P')-P')=0\] for all $\sigma \in \mathrm{Gal}(\overline{K}/K)$. Finally, $\Delta_1 P=nQ\in n G(K)$, so $\Delta_1[P]=[\Delta_1  P]=0$ in $G(K)/n G(K)$, which proves the lemma.
\end{proof}

We are now ready to prove the main result of this section.
\begin{proof}[Proof of Theorem \ref{main-GR}]
Let $P \in G(K)$. By Lemma \ref{quotient-finite-Remond}, the quotient group 
\[\left(G(K)\cap \left(\mathrm{End}_K(G) \cdot P\right)_{\mathrm{div}}\right)/\mathrm{End}_K(G) \cdot P\]
is finite. We denote its exponent by $b_P$.

For the remaining part of the proof,  let $[P]\in G(K)/n G(K)$ denote the class of $P$. Let also $P_n$, $\delta$, and $\mathrm{res}$ be as in the statement of Lemma \ref{lemma2-GR}. We regard $\mathrm{res}$ as a map
\[ \mathrm{res}: H^1\left(\mathrm{Gal}(\overline{K}/K),G[n]\right)\rightarrow \mathrm{Hom}\left(\mathrm{Gal}(\overline{K}/K(G[n])), G[n]\right).\]
We set \[M_P=(\mathrm{res}\circ\delta)([P])\left(\mathrm{Gal}(\overline{K}/K(G[n]))\right)\subset G[n].\]
Note that
\begin{equation}\label{eq:M_P}
M_P = \{\tau(P_n) - P_n \mid \tau \in \mathrm{Gal}(\overline{K}/K(G[n]))\} = \mathrm{Im}(\rho_{P,n}).
\end{equation}

For the second inclusion in the statement of the theorem, we note that, if $\varphi \in \mathrm{End}_K(G)$ with $\varphi(P) = 0$, then $\varphi(P_n) \in G[n]$. Therefore, if $\tau \in \mathrm{Gal}(\overline{K}/K(G[n]))$, then
\[ \varphi(\tau(P_n)-P_n) = \tau(\varphi(P_n)) - \varphi(P_n) = 0.\]
We conclude that $\rho_{P,n}(\tau) \in G_P[n]$ and so
\begin{equation}\label{eq:second-inclusion}
    \mathrm{Im}(\rho_{P,n}) \subset G_P[n].
\end{equation}

Furthermore, the subgroup $M_P$ is stable under the action of $\mathrm{Gal}(\overline{K}/K)$ since, for $\sigma \in \mathrm{Gal}(\overline{K}/K)$ and $\tau \in \mathrm{Gal}(\overline{K}/K(G[n]))$, we have that $\sigma \circ \tau \circ \sigma^{-1} \in \mathrm{Gal}(\overline{K}/K(G[n]))$ and one can check that
\[\sigma((\mathrm{res} \circ \delta)([P])(\tau)) = (\mathrm{res} \circ \delta)([P])(\sigma \circ \tau \circ \sigma^{-1}).\]
Therefore, the quotient map \[\alpha_P: G\rightarrow G/M_P\] defines a $K$-isogeny, i.e. an isogeny that is defined over $K$. We now prove that there exists another $K$-isogeny
\[\alpha'_P: G/M_P \to G\]
of exponent dividing some integer $\Delta_2$, which depends only on $G$ and $K$ and is independent of $P$, where we recall that the exponent of an isogeny is defined as the exponent of its kernel. 

Indeed, by \cite[Lemma A.4]{BHP}, there are at most finitely many connected commutative algebraic groups up to $K$-isomorphism (i.e. isomorphism defined over $K$) that admit a $K$-isogeny to $G$ (or, equivalently, from $G$, see Lemma \ref{fact-isogeny-degree}). For each representative $G'$ of such a $K$-isomorphism class, we arbitrarily choose a $K$-isogeny from $ G'$ to $G $ and set $ \Delta_2 $ to be the least common multiple of the exponents of the (finitely many) chosen isogenies, which depends only on $G$ and $K$. We deduce that there exists a $K$-isogeny $\alpha'_P$ from $G/M_P$ to $G$ of exponent dividing $\Delta_2$ as desired.

We now show that $\alpha=\alpha_P'\circ \alpha_P\in \mathrm{End}_K(G)$ satisfies
\begin{equation}\label{eq:M_P-sandwich}
    \Delta_2 \mathrm{Ker}(\alpha)\subset M_P\subset \mathrm{Ker}(\alpha).
    \end{equation}
Indeed, while the second inclusion is clear, for the first one note that, if $Q\in \mathrm{Ker}(\alpha)$, then $\alpha_P(Q)\in \mathrm{Ker}(\alpha_P')$, hence $\alpha_P(\Delta_2 Q)=\Delta_2 \alpha_P(Q)=0$, so $\Delta_2 Q\in M_P$.

Now, for $\sigma\in\mathrm{Gal}(\overline{K}/K(G[n]))$, we have
\[ (\mathrm{res}\circ \delta)([\alpha(P)])(\sigma)=\delta_{\alpha(P)}(\sigma)=\alpha(\delta_P(\sigma)) = \alpha((\mathrm{res} \circ \delta)([P])(\sigma))\in \alpha(M_P)=\{0\}\] and hence $(\mathrm{res}\circ \delta)([\alpha(P)])=0$.
From this, using Lemmas \ref{lemma1-GR} and \ref{lemma2-GR}, we obtain that $\Delta_1 \alpha(P)\in n G(K)$ for some integer $\Delta_1\geq 1$ that depends only on $G$ and $K$. It follows that $\Delta_1 \alpha(P) = nP'$ where $P' \in G(K) \cap (\mathrm{End}_K(G) \cdot P)_{\mathrm{div}}$.

Recalling the definition of $b_P$ from the beginning of the proof, we deduce that $b_P P' \in \mathrm{End}_K(G)\cdot P$ and so
\[b_P \Delta_1 \alpha(P) = n(b_P P') \in n\left(\mathrm{End}_K(G)\cdot P\right).\]
We can therefore write $b_P  \Delta_1 \alpha(P)=n \beta(P)$ with $\beta\in \mathrm{End}_K(G)$, which is equivalent to the fact that $P\in \mathrm{Ker}(b_P \Delta_1\alpha-n\beta)$.

Therefore we have $G_P\subset \mathrm{Ker}(b_P \Delta_1\alpha-n\beta)$, thus $G_P[n]\subset \mathrm{Ker}(b_P \Delta_1\alpha)$.
This, together with \eqref{eq:M_P} and \eqref{eq:M_P-sandwich}, implies that
\[b_P \Delta_1\Delta_2 G_P[n]\subset \Delta_2 \mathrm{Ker}(\alpha)\subset M_P = \mathrm{Im}(\rho_{P,n}).\]
Using this and \eqref{eq:second-inclusion}, we deduce the statement of the theorem with $\Delta=\Delta_1\Delta_2$.
\end{proof}

\section{On R\'emond's generalised Lehmer's conjecture}\label{Conjecture-1-rank0}
\subsection{Preliminaries and R\'emond's conjecture}
In  \cite{Rem}, R\'emond states a very strong conjecture on finite rank subgroups of tori and abelian varieties. We recall the notation in \cite[Notations 1.1]{Rem}. Let $G$ be a power of $\mathbb{G}_m$ or an abelian variety defined over a number field $K$, then to a certain type of ample line bundle $\mathcal{L}$ on a certain type of compactification $\overline{G}$ of $G$, we can attach a height on $G$ (see \cite[Section 2(c)]{Rem}). In our context, we take $\overline{G}=G$ and $\mathcal{L}$ symmetric if $G$ is an abelian variety and $\overline{G}=(\mathbb{P}^1)^n$  and $\mathcal{L}=\mathcal{O}(1,\ldots,1)$ if $G=\mathbb{G}_m^n$. In this way, we obtain the height $h_G$ on $G$ (as defined in Section \ref{intro}).

Let $\Gamma\subset {G(\overline{K})}$ be a subgroup, where $\overline{K}$ denotes a fixed algebraic closure of $K$. For us, a subvariety of $G$ is always irreducible and defined over $\overline{K}$ (but not necessarily over $K$).
Following \cite[Section 3(b)]{Rem}, we say that a subvariety of $G$ is \emph{$\Gamma$-torsion} if it is of the form $B+\gamma$ where $B$ is a connected algebraic subgroup of $G$ and $\gamma\in \Gamma_{\mathrm{sat}}$. We say that a subvariety $V\subsetneq G$ is \emph{$\Gamma$-transverse} if it is not contained in any proper $\Gamma$-torsion subvariety of $G$. Notice that the notions of $\Gamma$-torsion and $\Gamma$-transverse variety coincide with the classical notions of torsion variety and weak-transverse variety respectively when $\Gamma$ has rank 0, and with those of translate and transverse variety respectively when $\Gamma=G(\overline{K})$.

Before stating R\'emond's conjecture, we have to introduce another two objects attached to a subvariety $V\subset G$, defined over $\overline{K}$. The first is the \emph{essential minimum of $V$}, \[\mu_{\mathrm{ess}}(V)=\inf\left\{\theta\in\mathbb{R}\mid \text{ the set }\{x\in V(\overline{K})\mid h_G(x)\leq \theta\}\text{ is Zariski dense in }V\right\}.\] This object generalises, in some sense, the notion of height for varieties. Indeed, if $V=\{P\}$ is a singleton, then we have $\mu_{\mathrm{ess}}(V)=h_G(P)$. 

For a subvariety $V \subsetneq G$, the notion of degree relative to 
an algebraic extension $L \subset \overline{K}$ of $K$ is instead generalised by the \emph{obstruction index of $V$ relative to $L$}, defined as \[\omega_L(V)=\min_{Z\in \mathcal{F}}\left\{\deg_{\mathcal{L}}(Z)^{1/(\dim G-\dim Z)}\right\}\] where  $\mathcal{F}$ is the set of all 
equidimensional Zariski closed subsets  $Z\subsetneq G$ containing $V$ and defined over $L$. Here $\deg_{\mathcal{L}}(Z)=\mathcal{L}^{\cdot \dim Z}\cdot \overline{Z}=\mathcal{L}\cdots (\mathcal{L}\cdot(\mathcal{L}\cdot \overline{Z})))$ is defined as the  $\dim Z$-th iterate of the intersection product of $\mathcal{L}$ with the closure $\overline{Z}$ of $Z$ in $\overline{G}$. If $V=\{P\}$ is a point, then considering the full Galois orbit of $P$ via $\mathrm{Gal}(\overline{K}/L)$ yields that $\omega_L(V)\leq [L(P):L]^{1/\dim G}$.

Finally, we say that a subgroup $\Gamma\subset {G(\overline{K})}$ is a \emph{Lehmer group} if there exists $c_{\Gamma}>0$ such that  $\omega_{K(\Gamma)}(V)\mu_{\mathrm{ess}}(V)\geq c_{\Gamma}$ for every subvariety $V\subset G$ which is $\Gamma$-transverse. We can now state R\'emond's conjecture in its more general form:
\begin{con}[{\cite[Conjecture 3.4]{Rem}}]\label{con-Rem-lehmer}
    Every subgroup of $G(\overline{K})$ of finite rank is a Lehmer group.
\end{con}
This conjecture generalises several open problems. Indeed, as remarked in \cite[paragraph after D\'efinition 3.1]{Rem}, if $\Gamma=\{0\}$, then $\Gamma_{\mathrm{sat}}=G_{\mathrm{tor}}$ and $K(\Gamma)=K$, so Conjecture \ref{con-Rem-lehmer} implies the classical conjecture of Lehmer, while if instead $\Gamma=G_{\mathrm{tor}}$, it implies the relative version of Lehmer's conjecture. Notice that the definition of Lehmer group does not involve the rank of $\Gamma$. In particular, the question of whether $\Gamma=G(\overline{K})$ (which is not of finite rank) is a Lehmer group is related to the effective Bogomolov problem (see, for instance, the results in \cite{AD03,Ga}). 

It should then come as no surprise that Conjecture \ref{con-Rem-lehmer} remains largely open and that only a few specific cases of a weaker form of it are known. Such a weaker form is however enough to deduce Conjecture \ref{Rem-conj}, as shown in the next section. 

\subsection{A weak form of R\'emond's conjecture implies Conjecture \ref{Rem-conj}}\label{app-rem-conj}
In order to state the weak form of Conjecture \ref{con-Rem-lehmer} we are interested in, we need to give another definition.

We say that a subgroup $\Gamma\subset G(\overline{K})$ is a \emph{Dobrowolski group} if, for every $\varepsilon>0$, there exists $c_{\Gamma}(\varepsilon)>0$ such that  $\omega_{K(\Gamma)}(V)^{1+\varepsilon}\mu_{\mathrm{ess}}(V)\geq c_{\Gamma}(\varepsilon)$ for every subvariety $V\subset G$  which is $\Gamma$-transverse. Clearly, if $\Gamma$ is a Lehmer group, then it is a Dobrowolski group. Hence, one can consider the following weaker form of Conjecture \ref{con-Rem-lehmer}:
\begin{con}\label{con-Rem-dobro}
   Every subgroup of {$G(\overline{K})$ of finite rank} is a Dobrowolski group. 
\end{con}
Conjecture \ref{con-Rem-dobro} is also largely open.
As explained in \cite[Th\'eor\`eme 3.3, its proof, and Section 3(c)]{Rem}, some of the few known cases are when $\Gamma$ is either $\{0\}$ or $G_{\mathrm{tor}}$ and $G$ is a power of $\mathbb{G}_m$ (see for instance \cite{Dob,AZ,AD01,Del}) or an abelian variety with complex multiplication (see \cite{DH,Car}).

We now show how Conjecture \ref{con-Rem-dobro} implies our Conjecture \ref{Rem-conj}.
\begin{proposition}
Let $G$ be either a power of $\mathbb{G}_m$ or an abelian variety defined over a number field $K$ with a fixed algebraic closure $\overline{K}$. Let $\Gamma\subset G(\overline{K})$ be a finite rank subgroup. Then $\Gamma_{\mathrm{sat}}$ is also of finite rank and Conjecture \ref{con-Rem-dobro} for $G$, $\Gamma_{\mathrm{sat}}$, and $K$ implies Conjecture \ref{Rem-conj} for $G$, $\Gamma$, and $K$.
\end{proposition}

\begin{proof}
We know that $\mathrm{End}(G)$ is a finitely generated $\mathbb{Z}$-module, so if $\Gamma$ has finite rank, then $\Gamma_{\mathrm{sat}}$ also has finite rank.
By Conjecture \ref{con-Rem-dobro}, $\Gamma_{\mathrm{sat}}$ is a Dobrowolski group.
Let now $L/K(\Gamma_{\mathrm{sat}})$ be a finite extension and let $P\in G(L)\setminus \Gamma_{\mathrm{sat}}$.
Consider the subvariety $V=\{P\}$.  As $P\not\in \Gamma_{\mathrm{sat}}$ and as, by Remark \ref{rmk:sat-div-sat}, $(\Gamma_{\mathrm{sat}})_{\mathrm{sat}} = \Gamma_{\mathrm{sat}}$, $V$ is not $\Gamma_{\mathrm{sat}}$-torsion.

We now show that every irreducible component of an intersection of two $\Gamma_{\mathrm{sat}}$-torsion subvarieties is again $\Gamma_{\mathrm{sat}}$-torsion. Suppose that $\gamma_1 + B_1$ and $\gamma_2 + B_2$ are two $\Gamma_{\mathrm{sat}}$-torsion subvarieties such that $(\gamma_1 + B_1) \cap (\gamma_2 + B_2)$ is non-empty and let $f: G \to G/B_1 \times G/B_2 =: G'$ denote the homomorphism induced by the two quotient maps, then, by Lemma \ref{lem:section-homomorphism} (which, by an analogous and even simpler proof, also holds for powers of $\mathbb{G}_m$ instead of abelian varieties), there is a homomorphism $\varphi: G' \to G$ such that $f \circ \varphi$ is multiplication by some integer $N \geq 1$ on $f(G)$. Since $f^{-1}(\gamma_1 + B_1,\gamma_2 + B_2) \neq \emptyset$, it follows that
\[ f(\varphi(\gamma_1 + B_1, \gamma_2 + B_2)) = (N\gamma_1+B_1,N\gamma_2+B_2).\]
The point $\gamma' = \varphi(\gamma_1+B_1,\gamma_2+B_2)$ belongs to $\Gamma_{\mathrm{sat}}$ and $(\gamma_1 + B_1) \cap (\gamma_2 + B_2)$ is a union of irreducible components of the pre-image of $\gamma' + (B_1 \cap B_2)$ under multiplication by $N$, which proves the claim.

By what we have just shown, there exists a smallest $\Gamma_{\mathrm{sat}}$-torsion subvariety of $G$ containing $V$, which we denote by $H$ and which is defined over $K(\Gamma_{\mathrm{sat}})$ since every algebraic subgroup of $G$ is defined over $K(G_{\mathrm{tor}}) \subset K(\Gamma_{\mathrm{sat}})$. We define 
\[\omega_{K(\Gamma_{\mathrm{sat}})}^H(V)=\min_{Z\in \mathcal{F}_H}\left\{\left(\frac{\deg_{\mathcal{L}}Z}{\deg_{\mathcal{L}}H}\right)^{\frac{1}{\dim H-\dim Z}}\right\}\]
where  $\mathcal{F}_H$ is the set of all equidimensional Zariski closed subsets $Z\subsetneq H$ containing $V$ and defined over $K(\Gamma_{\mathrm{sat}})$. By \cite[Th\'eor\`eme 3.7]{Rem}, for every $\varepsilon>0$, there exists $c'_{\Gamma_{\mathrm{sat}}}(\varepsilon)>0$ (not depending on $V$ or $H$) such that 
\begin{equation}\label{eq-rem-thm-3.7}
    \left(\omega_{K(\Gamma_{\mathrm{sat}})}^H(V)\right)^{1+\varepsilon}(\deg_{\mathcal{L}}H)^{s(G) \varepsilon}\mu_{\mathrm{ess}}(V)\geq c'_{\Gamma_{\mathrm{sat}}}(\varepsilon)
\end{equation} where $s(G)=1$ if $G$ is a power of $\mathbb{G}_m$ and $s(G)=2$ if $G$ is an abelian variety. Since $V$ is a singleton, we have $\mu_{\mathrm{ess}}(V)=h_G(P)$. Moreover, since $\deg_{\mathcal{L}}H\geq 1$, we have \begin{align*}
     (\omega^H_{K(\Gamma_{\mathrm{sat}})}(V))^{1+\varepsilon}(\deg_{\mathcal{L}}H)^{s(G)\varepsilon} &\leq [K(\Gamma_{\mathrm{sat}})(P):K(\Gamma_{\mathrm{sat}})]^{\frac{(1+\varepsilon)}{\dim H}}(\deg_{\mathcal{L}}H)^{2\varepsilon-\frac{(1+\varepsilon)}{\dim H}} \\
     &\leq [K(\Gamma_{\mathrm{sat}})(P):K(\Gamma_{\mathrm{sat}})]^{\frac{(1+\varepsilon)}{\dim H}}\\
     &\leq [L:K(\Gamma_{\mathrm{sat}})]^{1+\varepsilon}
 \end{align*}
 for every $\varepsilon < 1/(2\dim G)$. Hence \eqref{eq-rem-thm-3.7} gives 
 \[h_G(P)\geq \frac{c'_{\Gamma_{\mathrm{sat}}}(\varepsilon)}{[L:K(\Gamma_{\mathrm{sat}})]^{1+\varepsilon}},\]
 which proves the result if we fix, for instance, $\varepsilon = 1/(4\dim G)$.
\end{proof}

We now show how the works of Amoroso-Zannier and Baker-Silverman yield an unconditional proof of the analogue of Conjecture \ref{Rem-conj} for products of powers of $\mathbb{G}_m$ and abelian varieties with CM in the case where $\Gamma$ is of rank zero.

\begin{cor}\label{cor:torus-x-cm}
Let $G$ be the product of the torus $T = \mathbb{G}^n_{m,K}$ and an abelian variety $A$ with CM, both defined over a number field $K$. Then the conclusion of Conjecture \ref{Rem-conj} holds for $G$ and $K$ in the case where $\Gamma$ has rank zero.
\end{cor}

\begin{proof}
After replacing $K$ by a finite extension $K'$ and $G$, $T$, and $A$ by their base changes to $K'$, we can assume that all geometric endomorphisms of $A$ are defined over $K'$. Since every abelian subvariety of $A$ is the image of some geometric endomorphism of $A$, also all abelian subvarieties of $A$ and all their geometric endomorphisms are defined over $K'$.

If $\dim A = 0$, then
\[K'(\Gamma_{\mathrm{sat}}) = K'((\mathbb{G}_m)_{\mathrm{tor}})\]
and the extension $K'((\mathbb{G}_m)_{\mathrm{tor}})/K'$ is abelian, so the conclusion of the corollary follows from \cite[Theorem 1.1]{AZ}. We therefore assume from now on that $\dim A > 0$.

Because of the Weil pairing and the fact that $A$ admits a polarisation over $K$, we have that
\[K'(\Gamma_{\mathrm{sat}}) = K'(G_{\mathrm{tor}}) = K'((\mathbb{G}_m)_{\mathrm{tor}},A_{\mathrm{tor}}) = K'(A_{\mathrm{tor}}).\]
Furthermore, the extension $K'(A_{\mathrm{tor}})/K'$ is abelian: indeed, by \cite[Corollary 2 on p.~502]{Serre_Tate}, the image of the $\ell$-adic Galois representation attached to a geometrically simple abelian subvariety of $A$ is commutative for every prime $\ell$ and one can then use the fact that any point in $A_{\mathrm{tor}}$ is a sum of torsion points of geometrically simple abelian subvarieties of $A$ whose order is a prime power.

Let now $L/K(\Gamma_{\mathrm{sat}})$ be a finite extension and let $P\in G(L)\setminus \Gamma_{\mathrm{sat}}$. Set $L'=K' L$. Write $P=(P_T,P_A)$, where $P_T\in T(L)$ and $P_A\in A(L)$, and recall that $h_G(P)=h(P_T)+h_{\mathcal{L}}(P_A)$ where $h(P_T)$ is the sum of the Weil heights of the coordinates of $P_T$ and  $h_{\mathcal{L}}$ is the N\'eron-Tate height on $A$ attached to an ample symmetric line bundle $\mathcal{L}$ on $A$.  Then the lower bound for $h_G(P)$ is obtained as follows. Since $P$ is not torsion, either $P_T$ or $P_A$ is not torsion either. If $P_T$ is not torsion, we use again \cite[Theorem 1.1]{AZ}. If $P_A$ is not torsion, we first choose a finite extension $K''/K'$ such that $L'=K''(A_{\mathrm{tor}}) \subset (K'')^{\mathrm{ab}}$ and then apply \cite[Theorem 0.1]{BS}.
\end{proof}

\end{document}